\documentclass[11pt,a4paper]{article}

\usepackage{amsmath}
\usepackage{amsfonts}
\usepackage{amssymb}
\usepackage{amsthm}
\usepackage{graphicx,cite}
\usepackage{palatino}
\usepackage{mathrsfs}
\usepackage{color}
\usepackage[colorlinks, citecolor=blue,linkcolor=red]{hyperref}

\newtheorem{thm}{Theorem}[section]

\newtheorem{theorem}{Theorem}[section]

\newtheorem{lemma}[thm]{Lemma}
\newtheorem{proposition}[thm]{Proposition}
\newtheorem{corollary}[thm]{Corollary}

\theoremstyle{definition}

\theoremstyle{remark}
\newtheorem{remark}[thm]{Remark}

\numberwithin{equation}{section}

\newcommand{\N}{\ensuremath{\mathbb{N}}}

\newcommand{\R}{\ensuremath{\mathbb{R}}}

\DeclareMathOperator{\diverg}{div}

\newcommand{\barg}[1]{\bigl(#1\bigr)}
\newcommand{\Barg}[1]{\Bigl(#1\Bigr)}

\newcommand{\Bbar}[1]{\Bigl|#1\Bigr|}
\newcommand{\bsqb}[1]{\bigl[#1\bigr]}
\newcommand{\Bsqb}[1]{\Bigl[#1\Bigr]}
\newcommand{\scp}[2]{\langle #1,#2 \rangle}

\newcommand{\Bigscp}[2]{\Big\langle #1, #2 \Big\rangle}

\newcommand{\dv}[1]{\,\mathrm{d}#1}

\def\N{{\mathbb N}}
\def\R{{\mathbb R}}

\title{Variational models for the interaction of surfactants with curvature --
existence and regularity of minimizers in the case of flexible curves}

\author{Christopher Brand
\footnote{Fakult\"at f\"ur Mathematik, Universit\"at Regensburg, Universit\"atsstrasse 31, 
93053 Regensburg, Germany}
 \and Georg Dolzmann
\footnotemark[1]
\and Alessandra Pluda
\footnote{Dipartimento di Matematica, Universit\`a di Pisa, Largo Bruno Pontecorvo 5, 56127 Pisa, Italy}
}

\begin{document}

\maketitle

\begin{abstract}
\noindent Existence and regularity of minimizers for a geometric variational problem is shown. The variational integral models an energy contribution of the interface between two immiscible fluids in the presence of surfactants and includes a Helfrich type contribution, a Frank type contribution and a coupling term between the orientation of the surfactants and the curvature of the interface. Analytical results are proven in a one--dimensional situation for curves.
\end{abstract}

\textbf{Keywords}: geometric variational problem, Helfrich energy, surfactants, curvature, existence, regularity.

\section{Introduction}
The importance of surfactants for the formation of interfaces between immiscible fluids has been recognized a long time ago. Even today, a  mathematical analysis of the evolution of such a system driven by the motion of the fluids, the elasticity of the interface, and the interplay of the curvature of the interface with the surfactants and their orientation is not feasible, see, however, \cite{Lengeler} for a model that combines classical membrane elasticity with fluid dynamics, but does not include a director field describing the surfactants. If one does not consider the motion of fluids, the situation is different, as we will explain later in the introduction, listing literature on the Helfrich functional. In this article, we investigate a very specific aspect of such a system, namely the interaction of the orientation of surfactant molecules with the curvature of the interface in a one--dimensional situation, that is, for interfaces that are given by curves in the plane,
a dynamical model for two-dimensional surfaces in three-dimensional space will be investigated in~\cite{BDMP}.
Inspired by Laradji and Mouritsen~\cite{LaradjiMouritsen2000} we study the functional 
\begin{align}\label{EnergyLM}
  E_{LM}(\gamma, \eta) &= \frac{1}{2}\int_\gamma (\kappa + \delta \diverg_\gamma \eta)^2 \dv{s} + \frac{\lambda}{2} \int_\gamma |\nabla_\gamma \eta|^2 \dv{s} + L(\gamma) = E(\gamma, \eta) + L(\gamma)\,,\\\;&\text{with}\; \lambda,\delta\in\mathbb{R}, \lambda>0\,,\nonumber
\end{align}
where the surfactants are modeled by a director field $\eta$; see Section~\ref{notation} for the precise definitions of the curvature $\kappa$ and the differential operators on curves; if $\gamma$ is a simple closed curve, then the functional is the geometric functional defined on the trace of $\gamma$. Thus the term $\diverg_\gamma\eta$ serves as a spontaneous curvature the existence of which has been postulated in models related to lipid bilayers and for the Helfrich model\cite{Helfrich}. Specifically, we focus on the static case and investigate the existence of minimizers of $E_{LM}$, where, in view of applications, one can augment the variational problem with additional constraints on the length of the curve, the enclosed volume, or the length of the surfactants. This case serves as a study for potential equilibrium states for a dynamical system driven by the $L^2$ gradient flow of the system, see Section~\ref{DynamicalProblem} for more information.

Motivated by this functional, it is the scope of this article to begin the analysis of geometric functionals which do not only involve a curve, or, more generally, a manifold $\Gamma$, and objects derived from it like its surface area or its mean curvature, but functionals which combine geometric properties of the manifold with an independent vector field defined on the manifold and which include the interaction of the manifold with this vector field. In our one--dimensional model, the first term contains the curvature $\kappa = - \diverg_\gamma \nu$ and assigns, for $\delta=1$ and in the presence of the constraint $|\eta|=1$, energy to the deviation of the orientation of the surfactants from the normal direction.  The second term is a Frank energy term that is common in models for liquid crystals and structured fluids.

It is important to note that the functional is invariant under some changes of variables and that the terms in the energy have specific scaling properties. The signed curvature $\kappa =\scp{ \partial_{ss}\gamma }{(\partial_s \gamma)^\perp}$ is invariant under orientation preserving changes of variables, and thus
\begin{align}\label{InvarianceReparametrization}
\forall \varphi \in H^2([0,2\pi];[0,2\pi])\text{ diffeo with }\varphi'>0\text{ on }[0,2\pi]\colon E_{LM}(\gamma, \eta) = E_{LM}(\gamma\circ \varphi, \eta\circ \varphi)
\end{align}
where $\varphi:[0,2\pi]\to[0,2\pi]$ is an orientation preserving reparametrization of $\gamma$, see Section~\ref{OperatorsCurves} for details. However, $\kappa$ does change its sign if the change of variables is not orientation preserving. Therefore a change of orientation of $\gamma$ has to be compensated by a change of sign in $\eta$. An alternative formulation of the energy is obtained by choosing a fixed normal field $\nu$ associated to $\Gamma$ and replacing $\kappa$ by $\vec{\kappa}=-(\diverg_\Gamma \nu)\nu$. The last object is again a geometric object that does not change upon any change of coordinates. 

Since the curvature has units of one over length and since the integral has units of length, a scaling argument shows that it is necessary to penalize the length of the curve in order to avoid a dilation to infinity. More precisely, for $R\in \R$, $R>0$,
\begin{align*}
E_{LM}(R \gamma, \eta) = \frac{1}{R}\Bsqb{  \frac{1}{2}\int_\gamma (\kappa + \delta \diverg_\gamma \eta)^2 \dv{s} + \frac{\lambda}{2} \int_\gamma |\nabla_\gamma \eta|^2 \dv{s}}  + R L(\gamma)\,.
\end{align*}
This observation motivates the third term in the energy $E_{LM}$. In the mathematical literature, the modeling and the discussion of variational models involving a coupling between the orientation of surfactants and the curvature of the interface started in~\cite{BartelsDolzmannNochettoM2NA2010,BDRN2012,BrandDolzmannPAMM2019} and was also investigated in~\cite{BrandThesis}.  Liquid crystals on deformable surfaces were also considered in~\cite{NitschkeReutherVoigtLCSurfaces2020}.

In this paper we restrict our attention to the minimization in the class of immersions. In contrast to the evolution problem for the gradient flow of the energy, where a natural initial configuration is given by an embedding and for which the flow stays embedded for a positive time, minimization in the class of embeddings may not lead to an embedded minimizer. Therefore we formulate the minimization in the class $H_{imm}^2$. It is an open problem to characterize the relaxation of $E_{LM}$, that is, to describe all pairs $(\gamma, \eta)\in H_{per}^2 \times H_{per}^1$ that are limits of sequences $(\gamma_k, \eta_k)_{k\in \mathbb{N}}\in H_{emb}^2 \times H_{per}^1$ with uniformly bounded energy.

Our model is inspired by and closely related to models for biological membranes as proposed by Canham~\cite{Canham}
and Helfrich~\cite{Helfrich}.
The Helfrich functional has the general form 
\begin{equation*}\label{Introeq:Helfrich} 
F_H(S)=\int_S \kappa_1 (H-H_0)^2 + \kappa_2 K \,\mathrm{d}\sigma,
\end{equation*}
where $S$ denotes a smooth surface in $\mathbb{R}^3$,
$H$ and  $K$ are the mean curvature and the Gau{\ss} curvature of $S$, respectively, and $\kappa_1,\kappa_2,H_0$ are constants. In particular 
 $\kappa_1$ and $\kappa_2$ are the relevant curvature--elastic moduli and
$H_0$ is the spontaneous curvature, originally introduced to allow for chemically different
sides of the bilayer. In this model,
the shape of the membrane is a minimizer of $F_H$ among a suitable class of surfaces.
In the last decades the study of the Helfrich functional has inspired a lot of work in the mathematical community.

There are several contributions on the minimization problem~\cite{Eichmann2020,Wojtowytsch},
even in the case of more than one surface~\cite{ChoksiVeneroni,ChoksiMorandottiVeneroni, BrazdaLussardiStefanelli}.
Moreover there is no lack of stability results~\cite{BernardWheelerWheeler,ElliottFritzGraham} 
and also the associated Dirichlet boundary value problem has been considered~\cite{DeckelnickDoemelandGrunau,Eichmann2019}.
The Helfrich functional can be interpreted as the singular limit of a suitable 
approximating functional defined on diffuse interfaces~\cite{BellettiniMugnai,Helmers2013,Helmers2015}.  In~\cite{LussardiPeletierRoeger,PeletierRoeger} an interfacial energy arising from the hydrophobic effect is taken into account and it is shown that lipid bilayers favour partial localization and display resistance to bending, stretching and fracture. 
Considerably less has been done concerning the associated evolution equations, but see, for example,~\cite{BarrettGarckeNuernberg,DallAcquaPozzi}.

The paper is organized as follows: In Section~\ref{notation} we introduce the notation used throughout the paper and summarize results we use in the proofs. We include a short discussion of differential operators on curves and discuss the definition of $E_{LM}$ and relations between $E_{LM}$ and geometric functionals. The proof of the existence of minimizers for the variational integral $E_{LM}$ is presented in Section~\ref{staticproblem} and the regularity in Section~\ref{Section:Regularity}, which, in fact, contains regularity for arbitrary critical points. The existence and regularity results include constraints on the length of the curve, the enclosed area, and the length of the surfactants. The concluding Section~\ref{DynamicalProblem} indicates possible extensions of our model to dynamic equations that arise as the $L^2$ gradient flow of the functional and the Appendix contains the derivation of the Euler-Lagrange equations.

\section*{Acknowledgments}
Christopher Brand acknowledges financial support through DFG GRK 1692 ``Curvature, Cycles, and Cohomology''.
The paper was carried out also
during a visit of Georg Dolzmann at University of Pisa
 supported by PRIN 2017
``Variational methods for stationary and evolution problems with singularities and interfaces''.

%%%%%%%%%%%%%%%%%%%%%%%%%%%%%%%%%%%%%%%%%%%%%%%%%%%%%%%%%%%%
\section{Notation and preliminary results}\label{notation}
%%%%%%%%%%%%%%%%%%%%%%%%%%%%%%%%%%%%%%%%%%%%%%%%%%%%%%%%%%%%

In this article we fix the orientation of a curve, and do not consider orientation reversing reparametrizations. Moreover, the energy $E_{LM}$ depends only on derivatives of $\gamma$ and $\eta$ and we need to introduce a normalization in order to obtain uniqueness results. Unless otherwise stated, we therefore assume the following hypotheses which we refer to as (H):
\begin{itemize}
 \item [(H1)] the constants $\delta$, $\lambda\in\mathbb{R}$ satisfy $\lambda>0$; in general, no assumption on the sign of $\delta$ is made; dependence of constants on $\delta$ and $\lambda$ is not indicated;
 \item [(H2)] the functions $(\gamma,\eta)$ are elements of $H^2_{imm}([0,2\pi];\R^2)\times H_{per}^1([0,2\pi];\R^2)$ defined below, $L(\gamma)$ denotes the length of $\gamma$, and
 \begin{equation}
 \vert \partial_x\gamma(x)\vert=\frac{L(\gamma)}{2\pi}\quad \text{for all}\,x\in [0,2\pi]\,;
 \end{equation}
 \item [(H3)] 
after a suitable translation it is assumed
 \begin{equation}\label{eq:MeanValGamma}
\int_{0}^{2\pi} \gamma\dv{x}=0\,;
\end{equation}
 \item [(H4)]
since only $\nabla_\gamma \eta$ is included in $E_{LM}$, it is supposed that
 \begin{equation}\label{eq:MeanValN}
 \int_{0}^{2\pi} \eta\dv{x}=0\,.
\end{equation}
\end{itemize}
Therefore an $L^2$--bound on the derivative of $\gamma$ and $\eta$ implies by Poincar\'e's inequality a corresponding $L^2$--bound on the functions themselves. 
The assumption (H4) on $\eta$ is not imposed in the presence of the constraint $\vert\eta\vert^2=1$ on $\eta$; if this constraint holds, then the $L^2$--norm of $\eta$ is bounded by the length of the curve.
We stress the fact that, because of the invariance property~\eqref{InvarianceReparametrization}
and the geometric nature of the problem, there can only be 
uniqueness up to reparametrization and up to rigid motions.
The  hypotheses (H2), (H3) and (H4) fix parametrisation and translations, but rotations are still allowed.

\subsection{Curves}
A regular curve is a differentiable curve $\gamma:[0,2\pi]\to\mathbb{R}^2$ with $\partial_x\gamma\neq 0$ and a plane curve is an element in the set $H_{imm}^2$ defined by 
\begin{align*}
 H_{imm}^2 = H_{imm}^2([0,2\pi];\R^2) = \{ \gamma\in H^2([0,2\pi];\R^2) \colon \gamma\text{ regular, }\gamma(0)=\gamma(2\pi),\,\partial_x\gamma(0) = \partial_x\gamma(2\pi) \}\,.
\end{align*}
All functions are extended by periodicity to $\R$ if needed and the set $H^2_{imm}$ is seen as 
an open  subset of the Banach space $H_{per}^2$ where 
\begin{align*}
H_{per}^k = H_{per}^k([0,2\pi];\R^2) = \{ \gamma\in H^k([0,2\pi];\R^2) \colon \partial_x^{\ell}\gamma(0)=\partial_x^{(\ell)}\gamma(2\pi)\text{ for }\ell=0,\ldots,k-1 \}\,,\quad k\in \mathbb{N}\,,
\end{align*}
if differentiability of operators is considered. The arc length derivative of a regular curve is given by $\partial_s \gamma = |\partial_x\gamma|^{-1}\partial_x\gamma$. The change of variables that leads to an arc length parametrization leaves the class of $H^2$ immersions invariant.
More generally, if a regular curve admits a parametrization of class $C^k$, then its reparametrization
by arc length is still of class $C^k$ (see~\cite[Theorem 1.2.11]{AbateTovena}).
The proof can be adapted to the Sobolev space $H^2$, for a sketch see the  Appendix.

The unit tangent vector is denoted by $\tau= \partial_x\gamma/|\partial_x\gamma|=\partial_s\gamma$ and the unit normal vector by $\nu=J\tau$ where $J$ is the counterclockwise rotation in the plane by $\pi/2$. The oriented curvature of a plane curve is the scalar function $\kappa:[0,2\pi]\to\mathbb{R}$ defined by $\partial_s\tau=\kappa\nu$ and the curvature vector $\vec{\kappa}$ is given by $\vec{\kappa}=\partial_s\tau=\partial_{ss}\gamma$. The length and the enclosed (signed) area of a differentiable curve are interpreted as functionals $L:H^2_{imm}\to [0,\infty)$ and $A:H^2_{imm}\to \R$ defined by
\begin{equation*}
L\colon \gamma\mapsto L(\gamma)=\int_0^{2\pi} \vert\partial_x{\gamma}(x)\vert\, \dv{x}=\int_\gamma 1\,\dv{s}\,,\; A:\gamma\mapsto A(\gamma) = -\frac{1}{2}\int_0^{2\pi}\scp{\gamma}{J\partial_x\gamma}\dv{x} = -\frac{1}{2}\int_\gamma\scp{\gamma}{\nu}\dv{s}\,.
\end{equation*}
Since $\eta$ represents the local average of surfactants, it is reasonable to introduce a constraint on the length of $\eta$ as well. For simplicity, we choose the pointwise constraint $|\eta|^2=1$ which we formulate with the mapping 
\begin{align*}
 S\colon H_{per}^1([0,2\pi];\R^2) \to H_{per}^1([0,2\pi])\,,\quad \eta \mapsto |\eta|^2 -1\,.
\end{align*}
Since $H_{per}^1$ is a Banach algebra, $S$ is defined on $H_{per}^1$. Our analysis uses the following theorem on the existence of Lagrange multiplicators~\cite[Theorem 26.1]{Deimling1985}. We denote the Fr\'echet derivative of a differentiable function $F:X\to Y$ by $F'$, the topological dual of a Banach space $X$ by $X^\ast$, and the adjoint operator for a linear and bounded map $T\in \mathcal{L}(X,Y)$ by $T^\ast \in \mathcal{L}( Y^\ast, X^\ast)$.

\begin{theorem}\label{Deimlingc9s26t1}
 Let $X$, $Y$ be real Banach spaces, $B(x_0,r)\subset X$, $\Phi:B(x_0,r)\to \R$ and $F:B(x_0,r)\to Y$ continuously differentiable, $F(x_0)=0$ and $R(F'(x_0))$ closed. Suppose also that 
 \begin{align*}
  \Phi(x_0) = \min\{ \Phi(x) \colon x\in B(x_0,r) \text{ and } F(x)=0\}\,.
 \end{align*}
Then there exist ``Lagrange multipliers'' $\lambda\in \R$ and $y^\ast \in Y^\ast$, not all zero, such that 
\begin{align*}
 \lambda \Phi'(x_0) + (F'(x_0))^\ast y^\ast = 0\;\text{in}\;X^\ast\,.
\end{align*}
If $R(F'(x_0))=Y$, then $\lambda \neq 0$. 
\end{theorem}

We say that a constraint is admissible if its Fr\'echet derivative is onto. In view of the embedding $H_{imm}^2\hookrightarrow C^{1,\alpha}$ for all $\alpha\in [0,1/2]$, there exists for all $\gamma_0\in H_{imm}^2$ an $r>0$ such that the ball with radius $r$ in $H_{per}^2$ is contained in $H_{imm}^2$. Therefore the Fr\'echet derivatives $L'(\gamma_0)$ and $A'(\gamma_0)$ are defined and one can verify the assumptions in Theorem~\ref{Deimlingc9s26t1}. 
%We write $L_0 = L(\gamma_0)$.

\begin{lemma}\label{lemma:AdmissibleConstraint}
Suppose that $\gamma_0\in H^2_{imm}([0,2\pi];\R^2)$, $|\partial_x\gamma_0|=L_0/2\pi$, and that $\eta_0\in H_{per}^1([0,2\pi];\R^2)$ with $|\eta_0|=1$ on $[0,2\pi]$. Then there exists an $r>0$ such that the functionals $L$, $A$, and $S$ define admissible constraints in the sense of Theorem~\ref{Deimlingc9s26t1} on $B(\gamma_0,r)\subset H_{per}^2$ which we refer to as (L), (A), and (S). Additionally, the functional $G:H_{imm}^2 \to \R^2$, $\gamma\mapsto (L(\gamma),A(\gamma))$ defines an admissible constraint unless $\gamma_0$ has constant curvature. 
\end{lemma}

\begin{proof} 
For simplicity we write $\mathrm{d}s$ for the arc length with respect to $\gamma_0$, we add the subscript $0$ to the geometric quantities related to $\gamma_0$, and we write $L_0 = L(\gamma_0)$.
An integration by parts shows that for all functions  $\varphi\in H^1_{per}([0,2\pi];\R^2)$ and $\eta\in H_{per}^1([0,2\pi];\R^2)$
\begin{align*}
L'(\gamma_0)[\varphi] = \int_0^{2\pi}\left\langle\tau_0,\partial_x\varphi\right\rangle \dv{x} = -\int_0^{2\pi}\left\langle \partial_x\tau_0,\varphi\right\rangle \dv{x}=-\int_{\gamma}\langle\kappa_0 \nu_0,\varphi\rangle \dv{s}\,,\quad S'(\eta_0)[\eta] = 2\langle \eta_0, \eta \rangle_{\R^2}
\end{align*}
while by the identity $J^T=-J$ and $\nu_0 = J \partial_x \gamma_0/|\partial_x \gamma_0|$
\begin{align*}
A'(\gamma_0)[\varphi] &= -\frac{1}{2} \int_0^{2\pi} \langle \varphi, J\partial_x \gamma_0 \rangle + \langle \gamma_0, J\partial_x \varphi\rangle\dv{x}\\ &= -\frac{1}{2} \int_0^{2\pi} \langle \varphi, J\partial_x \gamma_0 \rangle - \langle J^T\partial_x\gamma_0, \varphi\rangle\dv{x}%= -\int_0^{2\pi} \langle J\partial_x\gamma_0, \varphi \rangle\dv{x}
= -\int_\gamma \langle \nu_0, \varphi \rangle\dv{s}\,,
\end{align*}
see also~\cite[Lemma~2.2]{GageContemMath1986}. 

We first prove that the range of $G'(\gamma_0)$ on $H_{per}^1$ is one--dimensional if and only if $\kappa_0$ is constant. In view of the formulas for the Fr\'echet derivative, the assumption that $\kappa_0$ is constant is sufficient for the range of $G$ being one--dimensional. Conversely, suppose that the range is one--dimensional, that is, that there exists a $\lambda\in \R$ with 
\begin{align*}
 L'(\gamma_0)[\varphi] = \lambda A'(\gamma_0)[\varphi]\; \Leftrightarrow \; 
 \int_{\gamma}(\kappa_0 - \lambda) \scp{\nu_0}{\varphi}\dv{s} =  \int_{0}^{2\pi}(\kappa_0 - \lambda) \scp{J\partial_x \gamma_0}{\varphi}\dv{x} = 0 \quad \text{ for all }\varphi\in H_{per}^1\,. 
\end{align*}
Since $\gamma_0\in H_{imm}^2$ and $\partial_x\gamma_0\in H_{per}^1$, we may use for any $g\in C_c^\infty([0,2\pi])$ the test function
$\varphi = g\nu_0  =2\pi g L_0^{-1}J\partial_x \gamma_0$. By the fundamental lemma of the calculus of variations, one gets $\kappa_0 = \lambda$ a.e.

Suppose now that $\kappa_0$ is not constant. Denote for an integrable function $f$ its average by $\overline{f}$. We first show that there exists a function $f\in C_c^\infty((0,2\pi))$ with $\overline{f}=0$ and $\int_{0}^{2\pi}\kappa_0 f \dv{x}\neq 0$. To see this, suppose the assertion were not true. By definition (as $f(0)=f(2\pi)=0$), for all $f\in C_c^\infty((0,2\pi))$ the function $f'\in C_c^\infty((0,2\pi))$ has average zero and therefore by a variant of the fundamental lemma in the calculus of variations, often referred to as du Bois-Raymond Lemma, see~\cite[Lemma 2 on Page 10]{GelfandFomin1963},
\begin{align*}
 \int_{0}^{2\pi} \kappa_0 f' \dv{x} = 0 \quad \Rightarrow \quad \kappa_0=const\,,
\end{align*}
a contradiction. 

To prove the Lemma, it suffices to show that $G$ and $S$ define admissible constraints. To simplify constants we choose $\varphi=\nu_0=2\pi L_0^{-1}J\partial_x\gamma_0\in H_{per}^1$ and the last formula for $L'$ shows that $G'(\gamma_0)[\varphi] = G'(\gamma_0)[\nu_0]=\barg{ -\int_\gamma \kappa_0 \dv{s}, -L_0 }$.

Since $\kappa_0\in L^2$, this defines a vector in a double cone about the $e_2$ axis with opening angle less than $\pi$.

To prove that $G'$ is onto, it thus suffices to find a different choice for $\varphi$ which leads to a vector which is linearly independent, for example in a double cone about the $e_1$ axis with arbitrarily small opening angle. Fix $f\in C_c^\infty((0,2\pi))$ with $\overline{f}=0$ and $\int_{0}^{2\pi}\kappa_0 f \dv{x}\neq 0$. By definition of $G'$, $G'$ is defined for $\varphi = f \nu_0 = f 2\pi L_0^{-1}J\partial_x \gamma_0\in H_{per}^1$  with
\begin{align*}
 G'(\gamma_0)[f \nu_0] = \Bigl( -\int_\gamma \kappa_0 f\dv{s}, -\int_\gamma f \dv{s} \Bigr) = -\Bigl(\int_\gamma \kappa_0 f\dv{s}\Bigl) e_1\neq 0\,.
\end{align*}
%If the curve is closed and simple, then the total curvature is $2\pi$. Otherwise it is an integer multiple of $2\pi$ \cite{SullivanOW2008} and does not vanish either. 
Thus $ G'(\gamma_0)[f \nu_0] $ is parallel to the $e_1$ axis, does not vanish, and is not contained in the double cone about the $e_2$ axis that was determined in the first step. Since $H_{per}^2$ lies dense in $H_{per}^1$ and all expressions are linear in $\varphi$, the map $G'$ is also onto if restricted to its domain $H_{per}^2$.

Finally, if $|\eta_0|=1$, then the map $S'(\eta_0)[\cdot]$ is onto. Indeed, if $\psi\in H_{per}^1([0,2\pi])$ is given, then the function $\eta = (1/2) \psi \eta_0$ satisfies $S'(\eta_0)[\eta] = \psi$.
\end{proof}

\subsection{Differential operators on curves and the definition of {$\boldsymbol{E_{LM}}$}}\label{OperatorsCurves}
In this section, we present differential operators on manifolds, their special form on curves in the plane, and we define the notation used in the definition of $E_{LM}$. For a curve $\gamma\in H_{per}^2([0,2\pi]; \R^2)$ its trace $\Gamma\subset \R^2$ is defined by $\Gamma = \gamma([0,2\pi])$. If $\gamma\in H_{\mathrm{emb}}^2$ with 
\begin{align*}
H_{\mathrm{emb}}^2 = H_{\mathrm{emb}}^2([0,2\pi];\R^2) =\{\gamma\in  H_{imm}^2([0,2\pi];\R^2)\colon \gamma\text{ simple and closed}\}
\end{align*}
is an embedding, then $\Gamma$ is an embedded manifold for which $\gamma^{-1}$ is a chart and for which the usual geometric derivatives of scalar functions $f$ and arbitrary vector fields $\eta$ along $\Gamma$ are defined at $p=\gamma(x)$ by
\begin{align*}
\nabla_\Gamma f(p) = \frac{\partial_x(f\circ\gamma)(x)}{|\partial_x \gamma(x)|} \, \frac{\partial_x\gamma(x)}{|\partial_x \gamma(x)|}\,,\quad  \diverg_\Gamma \eta(p) =\Bigscp{  \frac{\partial_x(\eta\circ\gamma)(x)}{|\partial_x \gamma(x)|}}{ \frac{\partial_x\gamma(x)}{|\partial_x \gamma(x)|}}\,.
\end{align*}
Moreover, $\nabla_\Gamma \eta = e_1 \otimes \nabla_\Gamma \eta_1 + e_2\otimes \nabla_\Gamma \eta_2$ is a matrix which contains the gradients of the components as rows. Note that $\nabla_\Gamma$ is the dual operator to $\diverg_\Gamma$ in the following sense: suppose that $f\in H^1(\Gamma)$ and $\eta\in H^1(\Gamma;\R^2)$, then the formula for integration by parts holds, 
\begin{align*}%\label{integrationbyparts}
 \int_{\Gamma} f \diverg_\Gamma \eta \dv{s} = - \int_{\Gamma} \scp{\nabla_\Gamma f}{\eta}\dv{s} - \int_{\Gamma} \kappa f \scp{\eta}{\nu}\dv{s}\,.
\end{align*}
With these definitions in place, one defines the geometric functional 
\begin{align*}
   E_{LM}^g(\Gamma, \eta)
   &  = \frac{1}{2}\int_{\Gamma} (\kappa + \delta \diverg_\Gamma \eta)^2 \dv{s} + \frac{\lambda}{2} \int_{\Gamma} |\nabla_\Gamma \eta|^2 \dv{s} + L(\Gamma)\,.
\end{align*}

If $\gamma\in H_{imm}^2$ is merely an immersion, then $\Gamma$ is not necessarily a manifold, but still many geometric quantities  may be defined locally as well. In fact, since $H_{imm}^2 \hookrightarrow C^1$, for each $x\in [0,2\pi]$ there exists an $\varepsilon>0$ such that $\gamma$ restricted to $(x-\varepsilon, x+\varepsilon)$ is injective. Define $\Gamma_{x,\varepsilon} = \gamma((x-\varepsilon, x+\varepsilon))$. Then $\gamma|_{(x-\varepsilon,x+\varepsilon)}^{-1}:\Gamma_{x,\varepsilon}\to (x-\varepsilon, x+\varepsilon)$ is a chart, and if $f:\Gamma_{x,\varepsilon}\to \R$ is a function and $\eta:\Gamma_{x,\varepsilon}\to \R^2$ is a vector field, then $\nabla_\Gamma f$ and $\diverg_\Gamma \eta$ may be defined as before. 

This local representation of $\Gamma$ leads to a local definition of vector fields $\eta$ along $\Gamma=\gamma([0,2\pi])$ with $\gamma\in H_{imm}^2$.
In this case, $\eta$ is said to be a vector field along $\Gamma$ if it is defined by vector fields $\eta_{x,\varepsilon}$ on all the sets $\Gamma_{x,\varepsilon}$ and if for $(x-\varepsilon, x+\varepsilon) \cap (x'-\varepsilon', x'+\varepsilon')\neq \emptyset$ the compatibility condition $\eta_{x,\varepsilon} = \eta_{x',\varepsilon'}$ on $\gamma((x-\varepsilon, x+\varepsilon) \cap (x'-\varepsilon', x'+\varepsilon'))$ holds. Consequently, on $(x-\varepsilon, x+\varepsilon)$ the composition $\eta\circ\gamma$ is defined. For simplicity we write $\eta\circ \gamma$ without explicit reference to the local definition. We say that $\eta \in H^1(\Gamma)$ if it is a vector field along $\Gamma$ and $\eta\circ\gamma$ is of class $H^1$. With this local definition, one can extend $E_{LM}^g$ to $\Gamma = \gamma([0,2\pi])$ for curves that are not embeddings and vector fields $\eta$ along $\Gamma$ based on the local definition.

However, this local definition requires a local decomposition of $[0,2\pi]$ for a given curve $\gamma\in H_{imm}^2$ and is not well adapted to minimization problems. Therefore we use the following identification in the case of embeddings as a guideline for the definition of the functional $E_{LM}$ which consequently coincides with the corresponding functional using the usual geometric definitions in the case of embeddings. If $\gamma$ is an embedding with trace $\Gamma$, then there is a one-to-one correspondence between points $p\in \Gamma$ and $x\in [0,2\pi]$ and we can identify a function $f:\Gamma\to \R$ and a vector field $\eta:\Gamma\to \R$ with the function $\widetilde{f}:[0,2\pi]\to \R$, $x\mapsto \widetilde{f}(x) = (f\circ \gamma)(x)$ and the vector field $\widetilde{\eta}:[0,2\pi]\to \R^2$, $x\mapsto \widetilde{\eta}(x) = (\eta\circ \gamma)(x)$. We define $\int_\gamma f \dv{s} = \int_0^{2\pi}f(x)|\partial_x\gamma(x)| \dv{x}$, 
\begin{align*}
\nabla_\gamma \widetilde{f}(x) = \frac{\partial_x \widetilde{f}}{|\partial_x \gamma|}\frac{\partial_x\gamma}{|\partial_x \gamma|}\,,\quad \diverg_\gamma \widetilde{\eta}(x) = \Bigscp{\frac{\partial_x \widetilde{\eta}}{|\partial_x \gamma|}}{\frac{\partial_x\gamma}{|\partial_x \gamma|}}
\end{align*}
and we see that the formula for integration by parts now holds in the form 
\begin{align*}%\label{integrationbyparts}
 \int_{\gamma} \widetilde{f} \diverg_\gamma \widetilde{\eta} \dv{s} = - \int_{\gamma} \scp{\nabla_\gamma \widetilde{f}}{\widetilde{\eta}}\dv{s} - \int_{\gamma} \kappa \widetilde{f} \scp{\widetilde{\eta}}{\nu}\dv{s}\,.
\end{align*}
The definition of the energy $E_{LM}^g$ for a simple closed curve is now given in the chart $\gamma$ by 
\begin{align*}%\label{EnergyLM}
E_{LM}^g(\Gamma,\eta)  
&= \frac{1}{2}\int_0^{2\pi} \Barg{ \kappa(x) + \delta \Bigscp{\frac{\partial_x \widetilde{\eta}}{|\partial_x \gamma|}}{\frac{\partial_x\gamma}{|\partial_x \gamma|}}} ^2 |\partial_x \gamma|\dv{x} + \frac{\lambda}{2} \int_0^{2\pi} \Bbar{\frac{\partial_x \widetilde{\eta}}{|\partial_x \gamma|}\otimes\frac{\partial_x\gamma}{|\partial_x \gamma|}}^2 |\partial_x \gamma|\dv{x} + L(\gamma)\\&=  E_{LM}(\gamma, \widetilde{\eta})
\end{align*}
and this expression serves as definition of $E_{LM}$ in the case that $\gamma$ is not an embedding. The minimization for $E_{LM}$ is thus carried out on function spaces defined on $[0,2\pi]$. From now on, we write $f=\widetilde{f}$ and $\eta = \widetilde{\eta}$ in the definition of $\nabla_\gamma$ and $\diverg_\gamma$. In particular, if $\varphi \in C^1([0,2\pi];[0,2\pi])$ is a strictly increasing diffeomorphism and if $\widetilde{\gamma} = \gamma\circ\varphi$ and $\widetilde{f} = f \circ\varphi$, then at $p=\widetilde{\gamma}(y)=\widetilde{\gamma}(\varphi^{-1}(x)) = \gamma(x)$,
\begin{align*}
\nabla_{\widetilde\gamma} \widetilde{f}(y) = \frac{\partial_y \widetilde{f}}{|\partial_y \widetilde{\gamma}|}\frac{\partial_y\widetilde{\gamma}}{|\partial_y \widetilde{\gamma}|}(y) = \frac{\partial_x f}{|\partial_x \gamma|}\frac{\partial_x\gamma}{|\partial_x \gamma|}(\varphi(y)) = \frac{\partial_x f}{|\partial_x \gamma|}\frac{\partial_x\gamma}{|\partial_x \gamma|}(x) = \nabla_{\gamma}f(x)\,.
\end{align*}
The calculation for $\diverg_\gamma\eta$ is analogous and establishes~\eqref{InvarianceReparametrization}. Along the same lines, the local definition of $E_{LM}^g$ is equivalent to the definition of $E_{LM}$.

%%%%%%%%%%%%%%%%%%%%%%%%%%%%%%%%%%%%%%%%%%%%%%%%%%%%%%
%%%%%%%%%%%%%%%%%%%%%%%%%%%%%%%%%%%%%%%%%%%%%%%%%%%%%%
\section{Existence of minimizers}\label{staticproblem}
%%%%%%%%%%%%%%%%%%%%%%%%%%%%%%%%%%%%%%%%%%%%%%%%%%%%%%
%%%%%%%%%%%%%%%%%%%%%%%%%%%%%%%%%%%%%%%%%%%%%%%%%%%%%%

The existence of minimizers follows with the direct method in the calculus of variations. If $(\gamma_k,\eta_k)_{k\in \N}$ is a minimizing sequence, then, by~\eqref{InvarianceReparametrization}, we may assume that $|\dot{\gamma}_k| = L(\gamma_k)/2\pi$ and that $L(\gamma_k)$ is bounded by the energy. The Gau{\ss}-Bonnet theorem provides a uniform $L^2$--bound on the curvature which, under the hypothesis (H), for a plane curve, gives a uniform bound on the $H^2$--norm of $\gamma$. 
Weak compactness in this space together with the compact embedding into $H^1$ and the lower semicontinuity of the variational integral imply the assertion.

\begin{lemma}[Bounds on geometric quantities]\label{bound-geo-quantities}
Let $(\gamma,\eta)\in H^2_{imm}([0,2\pi];\mathbb{R}^2)\times 
H_{per}^1([0,2\pi];\mathbb{R}^2)$ satisfy the a~priori bound $E_{LM}(\gamma,\eta)\leq C_0<\infty$. Then
\begin{equation*}
\int_{\gamma}\kappa^2\dv{s}\leq \frac{2(\lambda+\delta^2)}{\lambda}C_0 = \widetilde{C}\quad
\text{and}\quad 
\widetilde{c} = \frac{2\pi^2\lambda}{C_0(\lambda+\delta^2)} \leq L(\gamma)\leq C_0\,.
\end{equation*}
\end{lemma}

\begin{proof}
For simplicity we assume $\delta\geq 0$, the proof for $\delta\leq 0$ is analogous. Moreover, if a plane curve is a simple closed curve with positive  orientation, then the total curvature is $2\pi$. More generally, if the curve is a plane curve, then the total curvature is an integer multiple of $2\pi$ \cite{SullivanOW2008} and does not vanish either. After changing the orientation of $\gamma$ and the sign of $\eta$, if needed, we may assume that the total curvature of $\gamma$ is greater than or equal to $2\pi$.

We begin by showing the bound on the $L^2$--norm of the curvature of $\gamma$. We first derive a lower bound on the quantity $(\kappa+\delta\diverg_\gamma \eta)^2$ in terms of $\kappa^2$ and $(\diverg_\gamma \eta)^2$. Using the generalized Young's inequality in the form $\vert ab\vert \leq (\varepsilon/2)a^2+1/(2\varepsilon)b^2$ with $a=\delta\kappa$, $b=\diverg_\gamma \eta$ and $\varepsilon=1/(\lambda+\delta^2)>0$ we get 
	\begin{equation}\label{step-with-young}
	\delta\kappa\diverg_\gamma \eta\geq -\frac{\delta^2}{2\left(\lambda+\delta^2\right)}\kappa^2 
	-\frac{\lambda+\delta^2}{2}\left(\diverg_\gamma \eta\right)^2\,.
	\end{equation}
Since $\Vert\diverg_{\gamma}\eta\Vert_{L^2}\leq \Vert\partial_s \eta\Vert_{L^2}$, we obtain for $E$ defined in~\eqref{EnergyLM}
	\begin{align}\label{estimate-product}
	\begin{aligned}	
	E(\gamma,\eta)
\geq &
	\int_\gamma \frac{\kappa^2}{2}+\frac{\delta^2}{2}(\diverg_{\gamma}\eta)^2+\delta\kappa\diverg_{\gamma}\eta
	+\frac{\lambda}{2}\vert\partial_s \eta\vert^2\dv{s}\\
\geq & 
	\int_{\gamma}\frac{\lambda \kappa^2}{2\left(\lambda+\delta^2\right)}
	-\frac{\lambda}{2}\left(\diverg_{\gamma}\eta\right)^2
	+\frac{\lambda}{2}\vert\partial_s \eta\vert^2\dv{s}
\geq \frac{\lambda}{2\left(\lambda+\delta^2\right)}\int_{\gamma}\kappa^2\dv{s}\,,
	\end{aligned}
	\end{align}
and hence the $L^2$--bound for $\kappa$. 
Since $E_{LM}= E + L$, the upper bound on $L$ follows from the assumption. To derive a lower bound, one uses Fenchel's Theorem, see~\cite[Section~5.7,~Theorem~3]{DoCarmoDiffGeo1976}, together with H\"{o}lder's inequality,\
	\begin{equation}\label{step-with-Hoelder}
	2\pi \leq \int_{\gamma} |\kappa|\dv{s}\leq \Bigl(\int_{\gamma}\kappa^2\dv{s}\Bigr)^{1/2}
	\Bigl(\int_\gamma 1\dv{s}\Bigr)^{1/2}\,,
	\end{equation}
that is,
	\begin{equation}\label{estimate-kappa-wrt-L}
	\int_{\gamma}\kappa^2\dv{s}\geq \frac{4\pi^2}{L(\gamma)}\,,
	\end{equation}
and together with the $L^2$--bound on the curvature~\eqref{estimate-product} one obtains $L(\gamma)\geq 4\pi^2/\widetilde{C}=\widetilde{c}>0$.
\end{proof}

\begin{lemma}[Lower bound on the energy]\label{lem:lower-bound-energy}
For $(\gamma,\eta)\in H^2_{imm}([0,2\pi];\mathbb{R}^2)\times 
H_{per}^1([0,2\pi];\mathbb{R}^2)$ the following inequality for $E_{LM}$ holds true, 
\begin{equation}\label{lowerboundenergy}
E_{LM}(\gamma,\eta)\geq 2\pi\sqrt{\frac{2\lambda}{\lambda+\delta^2}}\,.
\end{equation}
\end{lemma}

\begin{proof}
We conclude with the lower bounds on $E$ in~\eqref{estimate-product} and $\| \kappa \|_2^2$ in~\eqref{estimate-kappa-wrt-L} that
\begin{equation*}
	E_{LM}(\gamma,n)
	\geq \frac{\lambda}{2\left(\lambda+\delta^2\right)}\int_{\gamma}\kappa^2\dv{s}+L(\gamma)
	\geq \frac{2\lambda\pi^2}{\lambda+\delta^2}\cdot \frac{1}{L(\gamma)}+L(\gamma)
	\geq 2\pi\sqrt{\frac{2\lambda}{\lambda+\delta^2}}\,
	\end{equation*}
where we use in the last inequality the estimate $b/L + L\geq 2\sqrt{b}$ for $b\geq 0$.
\end{proof}

\begin{theorem}[Existence of minimizers for $E_{LM}$]\label{thm:existence}
Suppose that $\delta$, $\lambda\in\mathbb{R}$ and that $\lambda>0$. Then there exists a minimizer of the functional $E_{LM} \colon H^2_{imm}([0,2\pi];\mathbb{R}^2)\times H_{per}^1([0,2\pi];\mathbb{R}^2)\to \mathbb{R}$ subject to the constraints~\eqref{eq:MeanValGamma} and~\eqref{eq:MeanValN}.
\end{theorem}

\begin{proof}
The energy is nonnegative and finite for all $(\gamma,\eta)\in H^2_{imm}([0,2\pi];\mathbb{R}^2)\times H_{per}^1([0,2\pi];\mathbb{R}^2)$. Hence 
\begin{align*}
 \inf_{(\gamma,\eta)\in H_{imm}^2 \times H_{per}^1} E_{LM}(\gamma, \eta) = m\geq 0\,.
\end{align*}
Choose a minimizing sequence $(\gamma_k,\eta_k)_{k\in \N}$ in $H^2_{imm}([0,2\pi];\mathbb{R}^2)\times H_{per}^1([0,2\pi];\mathbb{R}^2)$ and assume that $E_{LM}(\gamma_k,\eta_k)\leq C_0$, $C_0 = 2m+1$, and that, by~\eqref{InvarianceReparametrization}, $|\partial_x\gamma_k| = L(\gamma_k)/2\pi$. Recall that we always assume~\eqref{eq:MeanValGamma} and~\eqref{eq:MeanValN}, i.e., the averages of $\gamma_k$ and $\eta_k$ vanish. By Lemma~\ref{bound-geo-quantities}, $L(\gamma_k)$ is uniformly bounded from above and below and we may assume that $L(\gamma_k)\to L_\infty\in(0,\infty)$ for $k\to \infty$. Since $|\partial_x \gamma_k|$ is constant in $x$,
\begin{align*}
|\kappa_k| = |\partial_{ss} \gamma_k| =\Bbar{ \frac{1}{|\partial_x\gamma_k|}\partial_x \Barg{ \frac{1}{|\partial_x\gamma_k|}\partial_x \gamma_k }} =  \frac{1}{|\partial_x\gamma_k|^2} \cdot | \partial_{xx}\gamma_k| \geq \frac{(2\pi)^2}{C_0^2} |\partial_{xx}\gamma_k|\,.
\end{align*}
By the bounds on $L(\gamma_k)$ and on the $L^2$--norm of $\kappa_k$ in Lemma~\ref{bound-geo-quantities} 
\begin{equation*}
\int_0^{2\pi} |\partial_x\gamma|^2 \dv{x} = 2\pi \cdot \frac{L(\gamma_k)^2}{(2\pi)^2} \leq \frac{C_0^2}{2\pi}
% \,,\quad 
% \int_0^{2\pi} |\partial_{xx}\gamma_k|^2 \dv{x} \leq \frac{C_0^4}{(2\pi)^4} \int_0^{2\pi}|\kappa|^2 \dv{x} \leq \frac{\widetilde C C_0^4}{(2\pi)^4}\,,
\end{equation*}
as well as 
$$
\int_{\gamma_k} \kappa_k^2\dv{s}_k=\int_0^{2\pi}\left(\frac{|\partial_{xx}\gamma_k|}{\vert\partial_x\gamma_k\vert^2}\right)^2\vert\partial_x\gamma_k\vert\dv{x}=
\int_0^{2\pi}\frac{\vert\partial_{xx}\gamma_k\vert^2}{\vert\partial_x\gamma_k\vert^3}\dv{x}
$$
and thus
$$
\int_0^{2\pi} |\partial_{xx}\gamma_k|^2 \dv{x}=\left(\frac{L(\gamma_k)}{2\pi}\right)^3\int_{\gamma_k}\vert\kappa_k|^2\dv{s}_k\leq \frac{\widetilde C C_0^3}{8\pi^3}\,.
$$
By Poincar\'e's inequality, which is applicable in view of~\eqref{eq:MeanValGamma}, we deduce the uniform bound 
$$\| \gamma_k\|_{H^2([0,2\pi];\mathbb{R}^2)}\leq C
$$ 
with a suitable constant $C<\infty$. Moreover,  \eqref{eq:MeanValN} holds,
\begin{align*}
 \int_{\gamma_k} |\nabla_{\gamma_k} \eta_k|^2 \dv{s}_k =  \int_{\gamma_k} |\nabla_{s_k} \eta_k \otimes \tau_k|^2 \dv{s}_k = \frac{1}{|\partial_x\gamma_k|}\int_0^{2\pi} |\partial_x \eta_k|^2 \dv{x}\geq \frac{2\pi}{C_0}\int_0^{2\pi} |\partial_x \eta_k|^2 \dv{x}
\end{align*}
and, again by Poincar\'e's inequality, $\| \eta_k \|_{H^1([0,2\pi];\mathbb{R}^2)}$ is uniformly bounded as well. Consequently there exists a subsequence $(\gamma_n,\eta_n)_{n\in \N} = (\gamma_{k_n},\eta_{k_n})_{n\in \N}$ such that $(\gamma_n, \eta_n)_{n\in\mathbb{N}}$ converges weakly in $H_{per}^2([0,2\pi];\mathbb{R}^2)\times H_{per}^1([0,2\pi];\mathbb{R}^2)$ to $(\gamma_{\infty},\eta_{\infty})$. By the compact embedding of Sobolev spaces into H\"older spaces one gets for every $\alpha\in (0,\frac{1}{2})$ the strong convergence in $C^{1,\alpha}([0,2\pi];\mathbb{R}^2)\times C^{\alpha}([0,2\pi];\mathbb{R}^2)$. In particular, the length functional is continuous with respect to convergence in $C^{1,\alpha}$ and we infer $L(\gamma_\infty) = L_\infty>0$  and $|\partial_x\gamma_\infty|=L_\infty/(2\pi)>0$. Thus $\gamma_\infty\in H^2_{imm}$. Note that 
\begin{align*}
 \kappa_n = \frac{1}{|\partial_x\gamma_n|^2} \cdot  \partial_{xx}\gamma_n\,,\quad \nabla_{\gamma_n} \eta_n = \frac{1}{|\partial_x\gamma_n|}\partial_x \eta_n \otimes \tau_n\,.
\end{align*}
Since $\gamma_n \to \gamma_\infty$ in $C^{1,\alpha}$, $\tau_n\to \tau_\infty$ in $C^{\alpha}$ and we find $\kappa_n \rightharpoonup \kappa_\infty$ in $L^2$ and $\nabla_{\gamma_n}\eta_n \rightharpoonup \nabla_{\gamma_\infty} \eta_\infty$ in $L^2$ for $n\to \infty$. Moreover, the sequences $| \partial_x \gamma_\infty|^{1/2}(\kappa_n + \delta\diverg_{\gamma_n}\eta_n):[0,2\pi]\to \R$ and $| \partial_x \gamma_\infty|^{1/2}\nabla_{\gamma_n} \eta_n:[0,2\pi]\to \R$ are weakly convergent in $L^2$ and uniformly bounded in $L^2$ by a constant $C_1$. Hence 
\begin{align*}
 \int_{\gamma_n} (\kappa_n +\delta \diverg_{\gamma_n}\eta_n)^2 &\dv{s}_n  = \int_0^{2\pi}(\kappa_n +\delta \diverg_{\gamma_n}\eta_n)^2 | \partial_x \gamma_n|\dv{x}\\ & \geq \int_0^{2\pi}(\kappa_n +\delta \diverg_{\gamma_n}\eta_n)^2 | \partial_x \gamma_\infty|\dv{x} - \int_0^{2\pi}(\kappa_n +\delta \diverg_{\gamma_n}\eta_n)^2 | \partial_x \gamma_n -  \partial_x \gamma_\infty|\dv{x}
\end{align*}
and in view of the convergence of $(\gamma_n)$ in $C^{1,\alpha}$ and lower semicontinuity of the norm with respect to weak convergence we find  for all $\varepsilon>0$ 
\begin{align*}
 \liminf_{n\to \infty} \int_{\gamma_n} (\kappa_n +\delta \diverg_{\gamma_n}\eta_n)^2 \dv{s}_n\geq \int_{\gamma_\infty}(\kappa_\infty +\delta \diverg_{\gamma_\infty}\eta)^2 \dv{s}_\infty - C_1 \varepsilon\,. 
\end{align*}
Therefore the first term in $E_{LM}$ is sequentially lower semicontinuous with respect to the given convergence, the argument for the second term is analogous, and the third term is in fact continuous. 
\end{proof}

We now consider the variational integral $E_{LM}$ subject to the constraints (A) and (S), that is, we seek for given $A_0\in\mathbb{R}$ minimizing pairs $(\gamma, \eta)\in H_{imm}^2([0,2\pi];\R^2) \times H_{per}^1([0,2\pi];\R^2)$ with $A(\gamma) = A_0$ and $S(\eta)=0$. Since $E_{LM}$ penalizes the length of the curve, we do not include the constraint on the length of the curve which, due to the isoperimetric
inequality, requires the condition $\vert A_0\vert\leq L(\gamma)^2/4\pi$. Recall that we do not impose the condition \eqref{eq:MeanValN}, that the average of $\eta$ vanishes, if the length of $\eta$ is prescribed.

\begin{corollary}[$E_{LM}$ with constraints]
Fix $\delta$, $\lambda$, $A_0\in \R$ with $\lambda>0$. Then the variational problem 
\begin{align*}
 \text{ minimize } E_{LM}\text{ in } \mathcal{A} = \{ (\gamma, \eta)\in H_{imm}^2([0,2\pi];\R^2) \times H_{per}^1([0,2\pi];\R^2)\colon A(\gamma) = A_0, \, S(\eta) = 0 \}
\end{align*}
has a solution.
\end{corollary}

\begin{proof}
If that $A_0\neq 0$, then let $\gamma$ be a circle with area $A_0$  (parametrized clockwise or counterclockwise depending on the sign of $A_0$) and let $\eta = \nu$ be a unit normal vector field. Then $(\gamma, \eta)\in \mathcal{A}$. 
If instead $A_0=0$, let $\gamma$ parametrize a figure eight and again let $\eta = \nu$ be a unit normal vector field. Also in this case $(\gamma, \eta)\in \mathcal{A}$. 
The assertion follows from the direct method in the calculus of variations applied to minimizing sequences $(\gamma_k, \eta_k)_{k\in \mathbb{N}}$ with $\gamma_k$ of vanishing mean value since the constraints (A) and (S) are continuous with respect to the convergence established in the proof of Theorem~\ref{thm:existence}, that is, strong convergence in $H_{per}^1\times L^2$ which implies, up to a further subsequence, convergence pointwise a.e.\ for $\eta_k$. Since $|\eta|=1$ almost everywhere, the $L^2$--bound for $\eta$ is immediate.
\end{proof}

Recall that $E$ defined in~\eqref{EnergyLM} does not include the penalization of the length. 

\begin{corollary}[$E$ with constraints]
Fix $\delta$, $\lambda$, $L_0$, $A_0\in \R$ with $\lambda>0$, $L_0>0$ and $A_0\in [-L_0^2/4\pi, L_0^2/4\pi]$. Then the variational problem 
\begin{align*}
 \text{ minimize } E\text{ in } \mathcal{A} = \{ (\gamma, \eta)\in H_{imm}^2([0,2\pi];\R^2) \times H_{per}^1([0,2\pi];\R^2)\colon L(\gamma) = L_0,\, A(\gamma) = A_0, \, S(\eta) = 0 \}
\end{align*}
has a solution.
\end{corollary}

\begin{proof}
For $\vert A_0\vert = L_0^2/4\pi$ let $\gamma$ parametrize 
(clockwise or counterclockwise depending on the sign of $A_0$)
a circle with area $A_0$ and let $\eta = \nu$ be a unit normal vector field. 
If $A_0=0$ let $\gamma$ be a figure eight with length $L_0$. 
In the case $0<\vert A_0\vert\leq L_0^2/4\pi$ let $\gamma$ parametrize an ellipse with length $L_0$ and area $A_0$ (again clockwise or counterclockwise depending on the sign of $A_0$) and choose $\eta=\nu$. In all cases, $(\gamma, \eta)\in \mathcal{A}$. The assertion follows from the direct method in the calculus of variations since the constraints (L), (A), (S) are continuous with respect to the convergence established in the proof of Theorem~\ref{thm:existence}.
\end{proof}

%%%%%%%%%%%%%%%%%%%%%%%%%%%%%%%%%%%%%%%%%%%%%%%%%%%%%%
\section{Regularity of critical points}\label{Section:Regularity}
%%%%%%%%%%%%%%%%%%%%%%%%%%%%%%%%%%%%%%%%%%%%%%%%%%%%%%

%Corollary~\ref{cor:circle} states that a parametrization of a circle $\gamma=\gamma(\delta,\lambda)$ and a normal vector field $\eta=\eta(\delta,\gamma)\nu$ provide a minimizer for $E_{LM}$. 
In view of the invariance under reparametrization~\eqref{InvarianceReparametrization}, the natural question concerning regularity addresses the regularity of solutions $(\gamma, \eta)$ for which $\gamma$ has been parametrized proportional to arc length. In this section we prove regularity of critical points, that is, for solutions of the Euler-Lagrange equations for $E_{LM}$ and for the corresponding necessary conditions for minima that result from the theorem on Lagrangian multipliers in the presence of some of the constraints (L), (A), (S) for $E_{LM}$ or $E$. If one considers $E$, one has to include at least the constraint (L). For completeness, the Euler-Lagrange equations for the functionals are derived in Lemma~\ref{ELEqn}. Set
\begin{align*}
 C_{per}^\infty([0,2\pi];\mathbb{R}^2) = \{ \gamma\in C^\infty([0,2\pi];\R^2) \colon \gamma^{(\ell)}(0) =\gamma^{(\ell)}(2\pi) \text{ for all }\ell\in \N_0 \}\,.
\end{align*}

\begin{remark}\label{RegularityGamma}
In this section, we focus on the regularity for $\kappa$ which implies regularity of $\gamma$. Indeed,  by the fundamental theorem of the local theory of curves in the plane, given $\kappa$ either of class $C^{k}([0,2\pi])$ or of class $W^{\ell,p}$, there exists, up to rigid motion of the plane, a unique regular curve $\widetilde{\gamma}$ either of class $C^{k+2}([0,2\pi])$ or of class $W^{\ell+2,p}$ parametrized by arc length 
with curvature $\kappa$
(see~\cite[page 19 and Ex.~9 page 24]{DoCarmoDiffGeo1976}).
Since we only consider critical points in $H^2_{imm}$, this curve coincides, up to a rigid motion,
with the given curve $\gamma$.
 Moreover, once regularity has been established in $W^{\ell+2,p}$, the same regularity follows in $W_{per}^{\ell+2,p}$ since we may extend all functions in the existence theorem by periodicity to $I=[-2\pi,4\pi]$ and argue on $I$.
\end{remark}

\begin{remark}
Notice that our results does not say anything on  regularity of the trace $\Gamma$. For example, we are not aware of sufficient conditions that guarantee that minimizers are simple closed curves.
\end{remark}

The proof of the regularity statement proceeds by duality,
as it is illustrated in the following lemma, that can be found for example in~\cite{AltFunctionalAnalysis2016}. Since the key quantity in the regularity statements is the term $\kappa + \delta \diverg_\gamma \eta$, we refer to $W^{k,\infty}$ regularity if this term is in $W^{k,\infty}$.

\begin{lemma}[{\cite[Corollary~6.13, Exercise~6.7]{AltFunctionalAnalysis2016}}]\label{LemmaAltDuality}
 Suppose that $\Omega\subset \R^n$ is open, $f\in L_{loc}^1(\Omega)$, $p\in (1,\infty]$, $1/p + 1/p'=1$, $m\in \mathbb{N}_0$, and that there exists a constant $C$ such that for all $k\in \mathbb{N}_0$ with $k\leq m$ and all $\zeta \in C_c^\infty(\Omega)$
 \begin{align*}
  \Bbar{ \int_\Omega f \partial^k \zeta \dv{x} } \leq C_0 \| \zeta \|_{L^{p'}(\Omega)}\,.
 \end{align*}
Then $f\in W^{m,p}(\Omega)$ and there exists a constant $C=C(m,C_0)$ with $\| f \|_{m,p}\leq C$.
\end{lemma}

\begin{proposition}[$L^\infty$ bounds for $E_{LM}$]\label{DualityArgumentLp}
Suppose that (H) holds. If a curve $\gamma\in H^2_{imm}([0,2\pi];\R^2)$ 
parametrized proportional to arc length together with a vector field $\eta \in H^1_{per}([0,2\pi];\R^2)$ is a critical point of $E_{LM}$, then $\gamma\in W_{per}^{2,\infty}([0,2\pi];\R^2)$ and $\eta\in W_{per}^{1,\infty}([0,2\pi];\R^2)$. Moreover, there exists a constant $C=C(\| \gamma \|_{H^2}, \| \eta \|_{H^1})$ with 
\begin{align*}
 \| \gamma \|_{W^{2,\infty}} + \| \eta \|_{W^{1,\infty}}\leq C(\| \gamma \|_{H^2}, \| \eta \|_{H^1})\,.
\end{align*}

\end{proposition}

\begin{proof}
By assumption $\vert\partial_x\gamma\vert=L(\gamma)/2\pi$. We first  prove 
$L^\infty$--regularity for the expression $\kappa+\delta \diverg\eta$. If $(\gamma, \eta)$ is a critical point of $E_{LM}$, then the first variation with respect to $\gamma$ vanishes and by~\eqref{variationELMgamma} for all $\varphi\in H_{per}^2$
\begin{align}
& \int_{\gamma} ( \kappa+\delta\diverg_\gamma \eta)\langle\partial_{ss}\varphi, \nu \rangle\dv{s} =\int_\gamma\Bigl(\frac{3}{2}\left(\kappa+\delta\diverg_\gamma \eta\right)^2 +\frac{\lambda}{2}\vert\partial_s\eta\vert^2-1\Bigr) \left\langle\tau,\partial_s\varphi\right\rangle\nonumber \\&-\delta\left(\kappa+\delta\diverg_\gamma \eta\right)\left\langle\partial_s\eta,\partial_s\varphi\right\rangle \dv{s}\,.\label{eq:definitionF}
\end{align}
We denote the right-hand side of~\eqref{eq:definitionF} by $F(\gamma,\eta,\varphi)$. Since $\gamma\in H_{imm}^2$, the $L^2$--norm of the curvature is bounded  and since $H^2\hookrightarrow C^{1}$ we find 
\begin{align*}
 \| \kappa+\delta\diverg_\gamma \eta \|_{L^1} \leq C(\gamma)  \| \kappa+\delta\diverg_\gamma \eta \|_{L^2}\leq C(\|\gamma\|_{H^2}, \|\eta\|_{H^1})
\end{align*}
and with H\"older's inequality in the last integral,
\begin{align*}
 |F(\gamma,\eta,\varphi)|\leq C \big( \| \kappa+\delta\diverg_\gamma \eta \|_{L^2}^2 + \| \partial_s \eta \|_{L^2}^2 + 1 \bigr) \| \partial_s \varphi\|_{L^\infty}\leq C(\|\gamma\|_{H^2}, \|\eta\|_{H^1}) \| \varphi \|_{W^{2,1}}\,.
\end{align*}
For $g\in C_{per}^\infty([0,2\pi])$ we define with $A$, $B\in \R^2$
\begin{align*}
 \varphi(x) = \frac{L(\gamma)^2}{(2\pi)^2} \int_0^x \int_0^y g(t)\nu(t) \dv{t}\dv{y} 
 + \left(\frac{L(\gamma)}{2\pi}x\right)A + \left(\frac{L(\gamma)^2}{8\pi^2} x^2\right)B
\end{align*}
where we choose $A$ and $B$ in such a way that $\varphi \in H_{per}^2$. The conditions are 
\begin{align*}
0 =  \varphi(0) & = \varphi(2\pi) = \frac{L(\gamma)^2 }{(2\pi)^2}\int_0^{2\pi} \int_0^y g(t)\nu(t) \dv{t}\dv{y} + L(\gamma)A + \left(\frac{L(\gamma)^2}{2}\right)B\,,\\
\frac{L(\gamma)}{2\pi}A =\partial_x \varphi(0) & = \partial_x \varphi(2\pi) = \frac{L(\gamma)^2 }{(2\pi)^2} \int_0^{2\pi} g(t)\nu(t) \dv{t} + \frac{L(\gamma)}{2\pi}A + \frac{L(\gamma)^2}{2\pi}B\,.
\end{align*}
The second equation determines $B$ and the first equation $A$ with 
\begin{align*}
 |B| &\leq \frac{1}{2\pi} \int_0^{2\pi} |g(t)|\dv{x}\leq C \| g \|_{L^1}\,,\\ |A|&\leq C(\|\gamma\|_{H^2}) \| g \|_{L^1} + C(\|\gamma\|_{H^2}) \int_0^{2\pi} \int_0^y |g(t)| \dv{t}\dv{y}\leq C(\|\gamma\|_{H^2})  \| g \|_{L^1}\,.
\end{align*}
Consequently $\varphi \in H_{per}^2$ with 
\begin{align}\label{defvarphi}
 \partial_{ss}\varphi = g\nu + B\,,\quad \| \varphi \|_{W^{2,1}}\leq C(\|\gamma\|_{H^2}) \| g \|_{L^1}\,.
\end{align}
We insert $\varphi$ in~\eqref{eq:definitionF} and find 
\begin{align*}
 \int_{\gamma} ( \kappa+\delta\diverg_\gamma \eta)\left( g + \langle B,\nu\rangle\right) \dv{s} = F(\gamma, \eta, \varphi)
\end{align*}
hence for all $g\in C_{per}^\infty([0,2\pi])$ with the estimate for $F(\gamma, \eta, \varphi)$
\begin{align}\label{EstimateLp}
& \int_{\gamma} ( \kappa+\delta\diverg_\gamma \eta) g\dv{s} \leq C(\|\gamma\|_{H^2}, \|\eta\|_{H^1}) \| \varphi \|_{W^{2,1}} + \| \kappa + \delta \diverg_\gamma \eta \|_{L^1} \| \scp{B}{\nu} \|_{L^\infty}\nonumber \\
 &\leq C(\|\gamma\|_{H^2}, \|\eta\|_{H^1}) \| g \|_{L^1}\,.
\end{align}
By duality, see Lemma~\ref{LemmaAltDuality},
\begin{align*}
 \kappa + \delta \diverg_\gamma\eta \in L^\infty\,,\quad \|  \kappa + \delta \diverg_\gamma\eta \|_{L^\infty}\leq C(\|\gamma\|_{H^2}, \|\eta\|_{H^1})\,.
\end{align*}
The variation~\eqref{variationELMeta} with respect to $\eta$ implies with $\lambda \neq 0$ and the $L^\infty$--bound just obtained 
\begin{equation}\label{eq:TestEq2}
\lambda \int_{\gamma} \left\langle\partial_s \eta, \partial_s\psi\right\rangle \dv{s}=
- \int_{\gamma} \delta (\kappa +\delta \diverg_\gamma\eta) \diverg\psi \dv{s} \leq C(\|\gamma\|_{H^2}, \|\eta\|_{H^1}) \| \psi \|_{W^{1,1}}\,.
\end{equation}
For $\sigma\in C_{per}^\infty([0,2\pi];\R^2)$ define 
\begin{align}\label{eq:Defpsi}
\psi(x)& = \frac{L(\gamma)}{2\pi} \int_0^x \sigma(t)\,\mathrm{d}t -\frac{L(\gamma)}{2\pi}Ax\,,\nonumber\\ 
A&=  \int_0^{2\pi}\sigma(t)\,\mathrm{d}t\,,\quad |A|\leq C\| \sigma\|_{L^1}\,,\quad  \| \psi \|_{W^{1,1}}\leq C(\|\gamma\|_{H^2}) \| \sigma \|_{L^1}\,.
\end{align}
We insert $\psi$ with $\partial_s \psi = \sigma-A$ in~\eqref{eq:TestEq2},
\begin{align}\label{EstimateEtaLp}
 \lambda \int_{\gamma} \langle\partial_s \eta,\sigma \rangle \dv{s} \leq  \lambda \int_{\gamma} \langle \partial_s \eta, A\rangle\dv{s} + C(\|\gamma\|_{H^2}, \|\eta\|_{H^1}) \| \sigma \|_{L^1} \leq C(\|\gamma\|_{H^2}, \|\eta\|_{H^1}) \| \sigma \|_{L^1}\,,
\end{align}
and conclude again by duality as in Lemma~\ref{LemmaAltDuality} that
\begin{align*}
 \partial_x \eta\in L^\infty\,,\quad \| \partial_x \eta \|_{L^\infty}\leq C(\|\gamma\|_{H^2}, \|\eta\|_{H^1})\,.
\end{align*}
Both estimates together imply $\kappa\in L^\infty([0,2\pi])$ and the assertion of the proposition in view of Remark~\ref{RegularityGamma}.
\end{proof}

\begin{proposition}[$W^{1,\infty}$ bounds for $E_{LM}$]\label{DualityArgumentW1p}
Suppose that (H) holds. If a curve $\gamma\in H^2_{imm}([0,2\pi];\R^2)$ 
parametrized proportional to arc length together with a vector field $\eta \in H_{per}^1([0,2\pi];\R^2)$ is a critical point of $E_{LM}$, then $\gamma\in W_{per}^{3,\infty}([0,2\pi])$ and $\eta\in W_{per}^{2,\infty}([0,2\pi];\R^2)$. Moreover, there exists a constant $C=C(\| \gamma \|_{H^2}, \| \eta \|_{H^1})$ with 
\begin{align*}
 \| \gamma \|_{W^{3,\infty}} + \| \eta \|_{W^{2,\infty}}\leq C(\| \gamma \|_{H^2}, \| \eta \|_{H^1}).
\end{align*}
\end{proposition}

\begin{proof}
We argue as in the proof of Proposition~\ref{DualityArgumentLp} and define for $g\in C_{per}^\infty([0,2\pi])$ with $L(\gamma)=\| \partial_x\gamma \|$, $A\in \R^2$
\begin{align*}
 \varphi(x) =\frac{L(\gamma)}{2\pi} \int_0^x g(t)\nu(t) \dv{t} - \frac{L(\gamma)}{2\pi}Ax\,,\quad A = \int_0^{2\pi}g(t)\nu(t)\dv{t}\,,\quad |A| \leq  C \| g \|_{L^1}\,.
\end{align*}
Since $\nu\in H_{per}^1$, the function $\varphi\in H_{per}^2$, $\partial_s \varphi = g \nu - A$ is an admissible test function with 
\begin{align*}
\langle\partial_{ss}\varphi, \nu \rangle = \langle \partial_s g \nu - \kappa g \tau, \nu \rangle =\partial_s g\,,\quad \| \varphi \|_{W^{1,1}}\leq C(\gamma) \| g \|_{L^1}\,,
\end{align*}
and \eqref{eq:definitionF} can be written as 
\begin{align*}
  \int_{\gamma} ( \kappa+\delta\diverg_\gamma \eta)\partial_sg\, \mathrm{d}s =\int_\gamma\Bigl(\frac{3}{2}\left(\kappa+\delta\diverg_\gamma \eta\right)^2 +\frac{\lambda}{2}\vert\partial_s\eta\vert^2-1\Bigr) \left\langle\tau,\partial_s\varphi\right\rangle -\delta\left(\kappa+\delta\diverg_\gamma \eta\right)\left\langle\partial_s\eta,\partial_s\varphi\right\rangle \dv{s}\,.
\end{align*}
By the $L^\infty$--bounds in Proposition~\ref{DualityArgumentLp}, 
\begin{align}\label{EstimateW1p}
   \int_{\gamma} ( \kappa+\delta\diverg_\gamma \eta)\partial_sg\, \mathrm{d}s \leq C( \| \gamma \|_{H^2}, \| \eta \|_{H1}) \| g \nu - A \|_{L^1}\leq  C( \| \gamma \|_{H^2}, \| \eta \|_{H1}) \| g \|_{L^1}\,.
\end{align}
By duality, see Lemma~\ref{LemmaAltDuality}, the two bounds~\eqref{EstimateLp} and~\eqref{EstimateW1p} imply $\kappa + \delta \diverg_\gamma\eta\in W^{1,\infty}$. This additional regularity allows us a partial integration in the right-hand side of~\eqref{eq:TestEq2}, and we find for all $\psi = \sigma\in C_{per}^\infty([0,2\pi];\R^2)$
\begin{align}\label{EstimateEtaH1p}
 \begin{aligned}
&\lambda \int_{\gamma} \left\langle\partial_s \eta, \partial_s\sigma\right\rangle \dv{s}=
- \int_{\gamma} \delta (\kappa +\delta \diverg_\gamma\eta) \diverg\sigma \dv{s}\\
&=\int_{\gamma} \delta\scp{ \nabla_s (\kappa +\delta \diverg_\gamma  \eta)}{\sigma} \dv{s} +\int_{\gamma} \delta (\kappa + \delta \diverg_\gamma\eta) \kappa \scp{\sigma}{\nu}\dv{s}\\
& \leq C(\delta,\Vert \kappa +\delta \diverg_\gamma  \eta\Vert_{W^{1,\infty}}, \Vert\nu\Vert_{L^\infty},\Vert\kappa\Vert_{L^\infty})\Vert \sigma\Vert_{L^1}\,.
\end{aligned}
\end{align}
The estimates~\eqref{EstimateEtaLp} and~\eqref{EstimateEtaH1p} imply by duality $\partial_s\eta \in W^{1,\infty}$, that is $\eta \in W^{2,\infty}$, and consequently $\diverg_\gamma \eta\in W^{1,\infty}$ and $\kappa\in W^{1,\infty}$. The assertion follows by Remark~\ref{RegularityGamma}.
\end{proof}

After these preparations we are in a position to prove full regularity. 

\begin{theorem}[Regularity for critical points of $E_{LM}$]\label{regularity-stationary-points}
Suppose that (H) holds. If a curve $\gamma\in H^2_{imm}([0,2\pi];\R^2)$ parametrized proportional to arc length together with a vector field $\eta \in H_{per}^1([0,2\pi];\R^2)$ is a critical point of the functional $E_{LM}$, then $\gamma,\eta\in C_{per}^{\infty}([0,2\pi];\R^2)$. Moreover, for all $k\in \N$ there exists a constant $C_k=C_k(\| \gamma \|_{H^2}, \| \eta \|_{H^1})$ with 
\begin{align*}
 \| \gamma \|_{W^{k+2,\infty}} + \| \eta \|_{W^{k+1,\infty}}\leq C_k(\| \gamma \|_{H^2}, \| \eta \|_{H^1}).
\end{align*}
\end{theorem}

\begin{proof}
The proof follows by induction based on Lemma~\ref{LemmaAltDuality}. Indeed, we prove that for all $m\in \mathbb{N}_0$, all $g\in C_{per}^\infty([0,2\pi])$, $\sigma\in C_{per}^\infty([0,2\pi];\R^2)$ and all $k\in \mathbb{N}_0$, $k\leq m$, 
\begin{align*}
  \int_{\gamma} ( \kappa+\delta\diverg_\gamma \eta)\partial_{s}^{k}g\dv{s} \leq C(\| \gamma\|_{H^2}, \| \eta \|_{H^1}) \| g \|_{L^1}\,,\quad 
  \lambda \int_{\gamma} \left\langle\partial_s \eta, \partial_{s}^{k}\sigma \right\rangle \dv{s} \leq C(\| \gamma\|_{H^2}, \| \eta \|_{H^1}) \| g \|_{L^1}\,.
\end{align*}
Then $\kappa + \delta \diverg_\gamma\eta\in W_{per}^{m,\infty}$, $\eta\in W_{per}^{m+1,\infty}$, $\kappa \in W_{per}^{m,\infty}$ and $\gamma\in W_{per}^{m+2,\infty}$ together with the corresponding estimates.

By~\eqref{EstimateLp}, \eqref{EstimateW1p}, \eqref{EstimateEtaLp}, \eqref{EstimateEtaH1p} the assertion holds for $m=1$ and Proposition~\ref{DualityArgumentW1p} states the corresponding regularity, $\gamma\in W_{per}^{3,\infty}$, $\kappa\in W_{per}^{1,\infty}$ and $\eta\in W_{per}^{2,\infty}$ together with the estimate. Suppose now that $m\geq 2$ and that the assertion holds for $m-1$. We need to establish the estimates for $k=m$ and assume that $\kappa+\diverg_\gamma\eta\in W_{per}^{k-1,\infty}$, $\eta\in W_{per}^{k,\infty}$, $\gamma\in W_{per}^{k+1,\infty}$, $\nu\in W_{per}^{k,\infty}$ together with the estimate. For $g\in C_{per}^\infty([0,2\pi])$ we use $\varphi = \partial_s^{k-2} g\nu$ as a test function in~\eqref{eq:definitionF} and calculate first \begin{align*}
 \langle \partial_{ss}\varphi, \nu \rangle & =  \langle \partial_{s}(\partial_s^{k-1} g \nu + \partial_s^{k-2}g \partial_s \nu), \nu \rangle = \partial_{s}^kg + 2  \langle \partial_{s}^{k-1}g \partial_s \nu, \nu \rangle +  \langle \partial_s^{k-2}g\partial_{ss}\nu, \nu \rangle \\ & = \partial_{s}^kg -  \langle\partial_s^{k-2} g\partial_{s}(\kappa\tau), \nu \rangle = \partial_{s}^kg - \kappa^2 \partial_s^{k-2}g\,.
\end{align*}
From~\eqref{eq:definitionF} we obtain with $\langle\partial_s \varphi, \tau \rangle = \langle \partial_s^{k-1} g \nu +\partial_s^{k-2} g \partial_s \nu, \tau\rangle = -\kappa \partial_s^{k-2}g$
\begin{align*}
& \int_{\gamma} ( \kappa+\delta\diverg_\gamma \eta)\partial_{s}^kg\dv{s}  = \int_{\gamma} \kappa^2 ( \kappa+\delta\diverg_\gamma \eta)\partial_s^{k-2}g\dv{s}\\ &\qquad  + \int_\gamma\Bigl(\frac{3}{2}\left(\kappa+\delta\diverg_\gamma \eta\right)^2 + \frac{\lambda}{2}\vert\partial_s\eta\vert^2-1\Bigr) (-\kappa) \partial_s^{k-2}g -\delta\left(\kappa+\delta\diverg_\gamma \eta\right)\left\langle\partial_s\eta,\partial_s^{k-1} g \nu + \partial_s^{k-2}g \partial_s \nu\right\rangle \dv{s}\,.
\end{align*}
In view of the regularity already established, we may integrate by parts in the terms involving derivatives of $g$ on the right-hand side and obtain together with the estimates that have been established
\begin{align}\label{EstimateHkp}
 \int_{\gamma} ( \kappa+\delta\diverg_\gamma \eta) \partial_{s}^k g \, \dv{s} \leq C(\|\gamma\|_{H^2}, \|\eta\|_{H^1}) \| g \|_{L^1}\,.
\end{align}
Since this estimate holds for all $k\leq m$ we conclude by duality $\kappa + \delta \diverg_\gamma\eta\in W^{m,\infty}$, see Lemma~\ref{LemmaAltDuality} together with the corresponding estimates.
Finally fix $\sigma\in C_{per}^\infty([0,2\pi];\R^2)$; we return to~\eqref{EstimateEtaH1p} and use the test function $\partial_s^{k-1} \sigma$ to obtain
\begin{align}\label{EstimateEtaHkp}
 \begin{aligned}
\lambda \int_{\gamma} &\left\langle\partial_s \eta, \partial_{s}^{k}\sigma\right\rangle \dv{s}=
- \int_{\gamma} \delta (\kappa +\delta \diverg_\gamma\eta) \diverg_\gamma\partial_s^{k-1}\sigma \dv{s}\\
&=\int_\gamma \delta\scp{ \nabla_s (\kappa +\delta \diverg_\gamma  \eta)}{\partial_s^{k-1}\sigma} \dv{s} +\int_\gamma \delta (\kappa + \delta \diverg_\gamma\eta) \kappa \scp{\partial_s^{k-1}\sigma}{\nu}\dv{s} \\ & \leq C(\Vert \kappa +\delta \diverg_\gamma  \eta\Vert_{W^{k,\infty}}, \Vert\nu\Vert_{W^{k-1,\infty}},\Vert\kappa\Vert_{W^{k-1,\infty}})\Vert \sigma\Vert_{L^1}\leq  C(\| \gamma \|_{H^2} + \| \eta \|_{H^1})\Vert \sigma\Vert_{L^1}\,.
\end{aligned}
\end{align}
Here all integrations by parts are justified on the right-hand side in view of the regularity already established. Since this estimate holds for all $k\leq m$, $\partial_s\eta\in W^{m,\infty}$, $\kappa\in W^{m,\infty}$ and therefore $\gamma\in W^{m+2,\infty}$, see Remark~\ref{RegularityGamma}.
\end{proof}

\begin{corollary}\label{cor:circle}
Fix $\delta$, $\lambda\in \R$ with $\lambda>0$. There exist $(\gamma, \eta)\in H^2_{imm}([0,2\pi];\mathbb{R}^2)\times H_{per}^1([0,2\pi];\mathbb{R}^2)$ with $\gamma$ parametrized by arc length which satisfy equality in~\eqref{lowerboundenergy}. In fact, any such $\gamma$ is a simple  curve and the trace of $\gamma$ is a circle of radius $\sqrt{\lambda}/\sqrt{2(\lambda+\delta^2)}$. The vector field $\eta$ is uniquely defined if one imposes the additional assumption~\eqref{eq:MeanValN}. In this case, it is a normal vector field given by $\eta = \delta/(\lambda+\delta^2)\nu$. Consequently the variational problem for $E_{LM}$ together with the constraints~\eqref{eq:MeanValGamma}
and~\eqref{eq:MeanValN} has a unique minimizer which is parametrized by arc length. 
\end{corollary}

\begin{proof}
As in Lemma~\ref{bound-geo-quantities} we assume that $\delta \geq 0$. If the functions $(\gamma,\eta)$ satisfy equality in~\eqref{lowerboundenergy}, then equality holds in all inequalities in the derivation of the lower bound for the energy. 
Equality in H\"older's inequality in~\eqref{step-with-Hoelder} implies that $|\kappa|$ is constant and since $\kappa$ is by Theorem~\ref{regularity-stationary-points} smooth, $\kappa$ is constant and, after a change of the orientation of $\gamma$ and of the sign of $\eta$, positive. Thus $\gamma$ defines a circle, possibly multiply covered, and equality in Youngs's inequality in~\eqref{step-with-young} leads in view of the lower bound to
\begin{align*}
 \frac{\delta}{\sqrt{\lambda+\delta^2}} \kappa = -\diverg_\gamma \eta \cdot \sqrt{\lambda+\delta^2}\quad \Leftrightarrow\quad \diverg_\gamma \eta = -\frac{\delta}{\lambda + \delta^2}\cdot \kappa\,.
\end{align*}
In particular, $\diverg_\gamma \eta$ is constant. In the last estimate in~\eqref{estimate-product} we find in case of equality that
\begin{align*}
|\partial_s \eta \cdot \tau| =  |\diverg_\gamma \eta| = |\nabla_\gamma \eta| = |\partial_s \eta \otimes \tau| = |\partial_s \eta|
\end{align*}
and $\partial_s \eta$ is parallel to $\tau$ with constant length. Minimization in $L(\gamma)$ in the lower bound in Lemma~\ref{lem:lower-bound-energy} implies
\begin{align*}
 \frac{1}{L(\gamma)^2} = \frac{\lambda+\delta^2}{2\lambda \pi^2}
\end{align*}
and with $n$ the number of coverings of the circle (with $n=1$ if $\gamma$ is a simple closed curve and therefore a circle)
\begin{align*}
 L(\gamma)^2 = (2n\pi R)^2 = \frac{2\lambda\pi^2}{\lambda+\delta^2}\quad \Leftrightarrow\quad R^2 = \frac{\lambda}{2n^2(\lambda + \delta^2)}\,.
\end{align*}
Since we assume $\kappa\geq 0$,
\begin{align*}
 \kappa = \frac{1}{R} = \sqrt{\frac{2n^2(\lambda+\delta^2)}{\lambda}}\,,\quad \frac{\delta \kappa}{\sqrt{\lambda + \delta^2}} =\frac{\delta \sqrt{2n^2(\lambda+\delta^2})}{\sqrt{\lambda + \delta^2}\sqrt{\lambda}} = | \diverg_{\gamma} \eta| \sqrt{\lambda+\delta^2}
\end{align*}
and hence
\begin{align*}
 |\diverg_{\gamma}\eta | =\frac{\sqrt{2n^2}\delta}{\sqrt{\lambda+\delta^2}\sqrt{\lambda}}  = \kappa \cdot \frac{\delta}{\lambda+\delta^2}\,.
\end{align*}
We compute the energy and find
\begin{align*}
E_{LM} & = \frac{1}{2}\int_\gamma (\kappa + \delta \diverg_\gamma \eta)^2 \dv{s} + \frac{\lambda}{2}\int_\gamma |\nabla_\gamma \eta|^2\dv{s} +L(\gamma)\\
& = \frac{1}{2} \int_\gamma \Bsqb{ \barg{ 1 - \frac{\delta^2}{\lambda+\delta^2}}\kappa }^2 \dv{s} + \frac{\lambda}{2} \int_\gamma \frac{2n^2\delta^2}{\lambda(\lambda+\delta^2)}\dv{s}+L(\gamma)\\
& = \frac{L(\gamma)}{2} \Barg{ \frac{\lambda}{\lambda+\delta^2} }^2 \cdot \frac{2n^2 (\lambda + \delta^2)}{\lambda} + \frac{L(\gamma)}{2} \cdot \frac{2n^2\delta^2}{\lambda+\delta^2}+L(\gamma)\\
&= \frac{L(\gamma)}{2} \Barg{\frac{\lambda}{\lambda+\delta^2} + \frac{\delta^2}{\lambda+\delta^2}} \cdot 2n^2+L(\gamma) = (n^2+1)L(\gamma) = (n^2+1)\pi\Barg{ \frac{2\lambda}{\lambda+\delta^2} }^{1/2}
\end{align*}
and hence $n=1$. Thus $\gamma$ is a circle. It remains to determine $\eta$. For simplicity we consider the arc length parametrization of $\gamma$. By assumption, $\tau$ and $\nu$ are in $H_{per}^1[0,L(\gamma)])$ and we may write $\eta = a \tau + b \nu$ with $a$, $b\in H_{per}^1([0,L(\gamma)])$. Since $\partial_s \eta = c_0\tau$, $c_0\in\mathbb{R}$,
the functions $a$ and $b$ satisfy
\begin{align*}
\partial_s \eta = (\partial_s a - \kappa b)\tau + (\kappa a + \partial_s b) \nu = c_0 \tau\quad \text{ or } \quad \left(\begin{array}{cl} 0 & 1 \\ 1 & 0  \end{array}\right)\left(\begin{array}{c} \partial_s a \\ \partial_s b   \end{array}\right) + \kappa \left(\begin{array}{c} a \\ -b  \end{array}\right) = \left(\begin{array}{c} 0 \\  c_0  \end{array}\right)\,.
\end{align*}
This is an inhomogeneous system of linear differential equations with constant coefficients and the general solution is given by 
\begin{align*}
 \left(\begin{array}{c} a(s) \\ b(s) \end{array}\right)  = c_1 \left(\begin{array}{c} \sin(\kappa s) \\ \cos(\kappa s)  \end{array}\right) + c_2 \left(\begin{array}{c} \cos(\kappa s) \\ -\sin(\kappa s)  \end{array}\right) + \left(\begin{array}{c} 0 \\ -c_0/\kappa \end{array}\right)
\end{align*}
and hence with $\gamma(s) = (1/\kappa) (\cos(\kappa s), \sin(\kappa s))$
\begin{align*}
 \eta &=-\frac{ c_0 }{\kappa} \nu + (c_1 \sin(\kappa s) + c_2 \cos(\kappa s))\left( \begin{array}{c} -\sin(\kappa s) \\ \cos(\kappa s) \end{array}\right) + (c_1 \cos(\kappa s) -c_2 \sin(\kappa s)) \left(\begin{array}{c} -\cos(\kappa s) \\ -\sin(\kappa s) \end{array}\right) \\&=-\frac{ c_0 }{\kappa} \nu- c_1 e_1 + c_2 e_2\,. 
\end{align*}
If~\eqref{eq:MeanValN} holds, then $c_1=c_2=0$ and $\eta = -(c_0/\kappa)\nu$. Since $\diverg_\gamma\eta = c_0 = -\delta\kappa/(\lambda + \delta^2)$, we conclude $c_0 =- \delta\kappa/(\lambda + \delta^2)$ and this is the assertion of the corollary.
\end{proof}

We consider now the regularity of critical points of the constrained problems.

\begin{proposition}[$L^\infty$--bounds for $E_{LM}$ with constraints]\label{DualityArgumentConstraintLp}
Suppose that (H) holds. If a curve $\gamma$ of class $H^2_{imm}([0,2\pi];\R^2)$ 
parametrized proportional to arc length together with a vector field $\eta \in H^1_{per}([0,2\pi];\R^2)$ is a critical point of $E_{LM}$ subject to the constraints (A) and (S), then $\gamma\in W_{per}^{2,\infty}([0,2\pi];\R^2)$ and $\eta\in W_{per}^{1,\infty}([0,2\pi];\R^2)$. Moreover, there exists a Lagrange multiplier $\psi^\ast\in H_{per}^{-1}([0,2\pi])$, $\psi^\ast= \psi_0^\ast + \partial_s \psi_1^\ast$ with $\psi_0^\ast$, $\psi_1^\ast\in L^2([0,2\pi])$ associated to the constraint (S) which satisfies $\psi_1^\ast\in L^\infty([0,2\pi])$ and there exists a constant $C=C(\| \gamma \|_{H^2}, \| \eta \|_{H^1}, \| \psi^\ast \|_{H^{-1}})$ with 
\begin{align*}
 \| \gamma \|_{W^{2,\infty}} + \| \eta \|_{W^{1,\infty}}  + \| \psi_1^\ast \|_{L^\infty}\leq C(\| \gamma \|_{H^2}, \| \eta \|_{H^1}, \| \psi^\ast \|_{H^{-1}})\,.
\end{align*}
\end{proposition}

\begin{proof}
Since $H_{imm}^2$ is an open subset in $H_{per}^2$ there exists an $r>0$ such that $B(\gamma,r)\subset H_{imm}^2\subset H_{per}^2$. Therefore we may consider $(\gamma, \eta)$ as a minimizer of
\begin{align*}
 E_{LM}\colon B_r= B((\gamma, \eta),r) \subset X \to \R\,,\quad X= H_{per}^2([0,2\pi];\R^2) \times H_{per}^1([0,2\pi];\R^2)
\end{align*}
subject to the given constraints. According to~\cite[Proposition~1.2]{AmbrosettiMalchiodi2007}, the map $E_{LM}$ is Fr\'echet differentiable in $B_r$ if the partial Fr\'echet derivatives with respect to $\gamma$ and $\eta$ are continuous in $B_r$. By~\eqref{variationELMgamma} and~\eqref{variationELMeta}, the partial Gateaux derivatives in the directions $\varphi\in H^2_{per}$ and $\psi\in H_{per}^1$ are given for $(\widetilde{\gamma},\widetilde{\eta})\in B_r$ by 
\begin{align*}
\partial_\gamma E_{LM}((\widetilde{\gamma},\widetilde{\eta}),\varphi)& = \int_{\widetilde{\gamma}}  (\widetilde{\kappa} + \delta \diverg_{\widetilde{\gamma}}\widetilde{\eta}) \bsqb{ \scp{\partial_{\widetilde{s}\widetilde{s}}\varphi}{\widetilde{\nu}} + \delta \scp{ \partial_{\widetilde{s}}\widetilde{\eta}}{ \partial_{\widetilde{s}} \varphi}  } \dv{\widetilde{s}} \\& + \int_{\widetilde{\gamma}} \Bigl(-\frac{3}{2} (\widetilde{\kappa}  + \delta\diverg_{\widetilde{\gamma}} \widetilde{\eta})^2 -\frac{\lambda}{2}\vert\partial_{\widetilde{s}}\widetilde{\eta}\vert^2+1 \Bigr)\left\langle\widetilde{\tau},\partial_{\widetilde{s}}\varphi\right\rangle \dv{\widetilde{s}} \,, \\ 
\partial_{\eta} E_{LM}((\widetilde{\gamma},\widetilde{\eta}),\psi) &= \int_{\widetilde{\gamma}} \delta (\widetilde{\kappa} +\delta \diverg_{\widetilde{\gamma}}  \widetilde{\eta}) \diverg_{\widetilde{\gamma}}\psi \dv{\widetilde{s}}  + \lambda \int_{\widetilde{\gamma}} \left\langle\partial_{\widetilde{s}} \widetilde{\eta}, \partial_{\widetilde{s}}\psi\right\rangle \dv{\widetilde{s}}\,.
\end{align*}
The Gateaux derivatives define bounded and linear functionals. 

Therefore the partial Fr\'echet derivatives exist, and by H\"older's inequality one sees that the Fr\'echet derivatives are continuous on $B((\gamma, \eta),r)\subset X$. Finally, by Lemma~\ref{lemma:AdmissibleConstraint} the constraints (A) and (S) are admissible constraints and we may use Theorem~\ref{Deimlingc9s26t1} with $\Phi = E_{LM}$, $B_r\subset X$ as constructed and the constraint
\begin{align*}
 F\colon X \to \R \times H_{per}^1([0,2\pi]) = Y\,,\quad (\varphi, \psi)\mapsto (A(\varphi), |\psi|^2-1)\,.
\end{align*}
Thus there exists a Lagrange multiplier $y^\ast \in Y^\ast$ such that $E_{LM}'(\gamma,\eta) + (F'(\gamma, \eta))^\ast y^\ast = 0$ in $X^\ast$. Since $Y^\ast = \R \times H_{per}^{-1}$, there exist $a\in \R$ and $\psi^\ast \in H_{per}^{-1}$ such that for all $(\varphi,\psi)\in X$ the identity 
\begin{align*}
&\langle E_{LM}'(\gamma,\eta),(\varphi, \psi) \rangle +\langle y^\ast,  F'(\gamma, \eta))(\varphi, \psi)\rangle\\
&=
 \partial_\gamma E_{LM}(\gamma,\eta)[\varphi] + \partial_\eta E_{LM}(\gamma,\eta)[\psi]
 + a  A'(\gamma)[\varphi] + \scp{\psi^\ast}{S'(\eta)[\psi]}=0
\end{align*}
holds. If one chooses $\psi=0$, then one finds for all $\varphi\in H_{per}^2$ that the equation that corresponds to~\eqref{eq:definitionF} has an additional term on the right-hand side, 
\begin{align*}
 \int_{\gamma} ( \kappa+\delta\diverg_\gamma \eta)\langle\partial_{ss}\varphi, \nu \rangle\dv{s} =\int_\gamma\Bigl(\frac{3}{2}\left(\kappa+\delta\diverg_\gamma \eta\right)^2 +\frac{\lambda}{2}\vert\partial_s\eta\vert^2-1\Bigr) \left\langle\tau,\partial_s\varphi\right\rangle\\ -\delta\left(\kappa+\delta\diverg_\gamma \eta\right)\left\langle\partial_s\eta,\partial_s\varphi\right\rangle -  a \scp{\nu}{\varphi}\dv{s}\,.
\end{align*}
The additional term is of lower order compared to the other terms on the right-hand side since $\nu=J\tau$ has the same regularity as $\tau$ and the arguments in the proof of Proposition~\ref{DualityArgumentLp} imply that $\kappa + \delta \diverg_\gamma\eta\in L^\infty$. The choice of $\varphi=0$ leads with~\eqref{eq:TestEq2} and $\psi^\ast = \psi_0^\ast + \partial_s\psi_1^\ast$, $\psi_0^\ast$, $\psi_1^\ast\in L^2([0,2\pi])$, to 
\begin{align}\label{BasicLagrangeMultiplierEta}
\begin{aligned}
\lambda \int_{\gamma} \left\langle\partial_s \eta, \partial_s\psi\right\rangle \dv{s} & =
- \int_{\gamma} \delta (\kappa +\delta \diverg_\gamma\eta) \diverg_\gamma\psi \dv{s} -  \scp{\psi^\ast}{S'(\eta)[\psi]}\\ & =
- \int_{\gamma} \delta (\kappa +\delta \diverg_\gamma\eta) \diverg_\gamma\psi \dv{s} -2 \int_{\gamma} \psi_0^\ast \scp{\eta}{\psi} - \psi_1^\ast \partial_s \scp{\eta}{\psi} \mathrm{d}s \,.
\end{aligned}
\end{align}
We expand the derivative and rearrange terms 
\begin{align}\label{PsiOneAstLInfty}
\int_{\gamma} \left\langle\lambda \partial_s \eta -2 \psi_1^\ast \eta, \partial_s\psi\right\rangle \dv{s} 
% & =- \int_{\gamma} \delta (\kappa +\delta \diverg_\gamma\eta) \diverg\psi_\gamma \dv{s} -  \scp{\psi^\ast}{S'(\eta)[\psi]}\\ 
& =
- \int_{\gamma} \delta (\kappa +\delta \diverg_\gamma\eta) \diverg_\gamma\psi \dv{s} -2 \int_{\gamma} \psi_0^\ast \scp{\eta}{\psi} - \psi_1^\ast \scp{\partial_s\eta}{\psi} \mathrm{d}s\,,
\end{align}
and use $\psi$ as an anti-derivative of $\sigma$ as in~\eqref{eq:Defpsi} as a test function. Since $\| \psi\|_{L^\infty}\leq C \| \psi \|_{W^{1,1}}\leq C(\| \gamma\|_{H^2}) \| \sigma \|_{L^1}$, we obtain for all $\sigma\in C_{per}^\infty([0,2\pi];\R^2)$
\begin{align*}
 \int_{\gamma} \left\langle\lambda \partial_s \eta -2 \psi_1^\ast \eta, \sigma\right\rangle \dv{s} \leq  \int_{\gamma} \left\langle\lambda \partial_s \eta -2 \psi_1^\ast \eta, A\right\rangle \dv{s} +   C(\gamma, \eta, \psi^\ast) \| \sigma \|_{L^1}\\ \leq   C(\| \gamma \|_{H^2}, \| \eta \|_{H^1}, \| \psi^\ast \|_{H^{-1}}) \| \sigma \|_{L^1}
\end{align*}
and by duality $h=\lambda \partial_s \eta -2 \psi_1^\ast \eta$ defines an element in $L^\infty$ and the norm is bounded by the constant on the right-hand side. By assumption, $\eta$ satisfies the constraint $|\eta|^2=1$ and $\langle \partial_s \eta, \eta\rangle = 0$. Consequently $\partial_s \eta\in L^\infty$ and, as in Proposition~\ref{DualityArgumentLp}, $\kappa\in L^\infty$ and $\gamma\in W_{per}^{2,\infty}$. This argument also proves that $\psi_1^\ast = -(1/2)\langle h,\eta\rangle\in L^\infty$ together with an estimate. 
\end{proof}

\begin{proposition}[$W^{1,\infty}$--bounds for $E_{LM}$ with constraints]\label{DualityArgumentConstraintW1p}
Suppose that (H) holds. If a curve $\gamma\in H^2_{imm}([0,2\pi];\R^2)$ 
parametrized proportional to arc length together with a vector field $\eta \in H_{per}^1([0,2\pi];\R^2)$ is a critical point of $E_{LM}$ subject to the constraints (A) and (S), then $\gamma\in W_{per}^{3,\infty}([0,2\pi])$, $\eta\in W_{per}^{2,\infty}([0,2\pi])$, $\psi_0^\ast \in L^\infty$, and $\psi_1^\ast\in W^{1,\infty}$. Moreover, there exists a constant $C=C(\| \gamma \|_{H^2}, \| \eta \|_{H^1})$ with 
\begin{align*}
 \| \gamma \|_{W^{3,\infty}} + \| \eta \|_{W^{2,\infty}} + \| \psi_0 \|_{L^\infty} + \| \psi_1^\ast \|_{W^{1,\infty}}\leq C(\| \gamma \|_{H^2}, \| \eta \|_{H^1}, \| \psi^\ast \|_{H^{-1}}).
\end{align*}
\end{proposition}

\begin{proof}
The regularity $\partial_s\eta\in L^\infty$ shown in Proposition~\ref{DualityArgumentConstraintLp} implies as in the proof of Proposition~\ref{DualityArgumentW1p} that $\kappa + \delta\diverg_\gamma \eta\in W^{1,\infty}$. With this information, one infers from~\eqref{PsiOneAstLInfty} with $\psi = \sigma\in C_{per}^\infty([0,2\pi];\R^2)$ after an integration by parts in the first term on the right-hand side that $\partial_s(\lambda \partial_s \eta -2 \psi_1^\ast \eta)\in L^\infty$, that is, $h\in W_{per}^{1,\infty}$. Thus $\psi_1^\ast = -(1/2) \langle h,\eta\rangle\in W_{per}^{1,\infty}$ and $\partial_s \psi_1^\ast\in L^\infty$. Consequently, $\lambda \partial_s \eta = h + 2 \psi_1^\ast \in W_{per}^{1,\infty}$ and we may rewrite~\eqref{BasicLagrangeMultiplierEta} as 
\begin{align*}
2\int_{\gamma} \psi_0^\ast \scp{\eta}{\psi}\dv{s} = \int_{\gamma}\lambda  \left\langle\partial_{ss} \eta, \psi\right\rangle- \delta\langle\nabla_\gamma (\kappa +\delta \diverg_\gamma\eta), \psi\rangle 
+\delta\left(\kappa +\delta \diverg_\gamma\eta\right)\kappa\left\langle\psi,\nu\right\rangle
 -2  \partial_s \psi_1^\ast \scp{\eta}{\psi} \mathrm{d}s \,.
\end{align*}
The special choice $\psi = g \eta$ with $g\in C_{per}^\infty([0,2\pi])$ implies that
\begin{align*}
 \int_{\gamma} \psi_0^\ast g \dv{s}\leq C(\| \gamma \|_{H^2}, \| \eta \|_{H^1}, \| \psi^\ast \|_{H^{-1}}) \| g \|_{L^1}
\end{align*}
and this estimate implies $\psi_0^\ast\in L^\infty$ together with the estimate.
\end{proof}

\begin{theorem}[Regularity for critical points of $E_{LM}$ and $E$ with constraints]\label{RegularityELM}
 Fix $\delta$, $\lambda$, $L_0$, $A_0\in \R$ with $\lambda>0$, $L_0>0$  and assume that $\gamma\in H^2_{imm}([0,2\pi];\R^2)$ is parametrized proportional to arc length and that $\eta \in H_{per}^1([0,2\pi];\R^2)$. If $(\gamma,\eta)$ is 
 \begin{itemize}
 \item[(i)] a critical point of $E_{LM}$ subject to the constraints (A) and (S) or
 
 \item[(ii)] a critical point of $E$ subject to the constraints (L), (A) and (S) with $\kappa$ not constant, $A_0\in [-L_0^2/4\pi,L_0^2/4\pi]$,
 \end{itemize}
then $\gamma$, $\eta\in C_{per}^{\infty}([0,2\pi];\R^2)$, $\psi_0^\ast$, $\psi_1^\ast\in C_{per}^{\infty}([0,2\pi])$. There exists a constant $$C_k=C_k(\| \gamma \|_{H^2}, \| \eta \|_{H^1}, \| \psi \|_{H^{-1}})\quad k\in \N\,,$$  with 
\begin{align*}
 \| \gamma \|_{W^{k+2,\infty}} + \| \eta \|_{W^{k+1,\infty}}  + \| \psi_0^\ast \|_{W^{k,\infty}}+ \| \psi_1^\ast \|_{W^{k,\infty}}\leq C_k(\| \gamma \|_{H^2}, \| \eta \|_{H^1}, \| \psi \|_{H^{-1}})\,.
\end{align*}
\end{theorem}
 
\begin{proof}
(i) The proof follows by induction as in Theorem~\ref{regularity-stationary-points} and we sketch the key estimate which states that for all $k\in \N$, $k\geq 1$, $\kappa\in W_{per}^{k,\infty}$, $\gamma\in W_{per}^{k+2,\infty}$, $\eta\in W_{per}^{k+1,\infty}$, $\psi_0^\ast\in W_{per}^{k-1,\infty}$, $\psi_1^\ast\in W_{per}^{k,\infty}$ together with the corresponding estimates. The case $k=1$ is stated in Proposition~\ref{DualityArgumentConstraintW1p}. Suppose thus that the assertion holds for $k-1\geq 1$. We use $\varphi = \partial_s^{k-2}g$ and $\psi = \partial_s^{k-1}\sigma$ with $g\in C_{per}^{\infty}([0,2\pi])$ and $\sigma\in C_{per}^{\infty}([0,2\pi];\R^2)$ as test functions as in Proposition~\ref{DualityArgumentConstraintW1p}. This choice implies $\kappa + \delta \diverg_\gamma \eta\in W_{per}^{k,\infty}$ and $h=\lambda \partial_s \eta -2 \psi_1^\ast \eta\in W_{per}^{k,\infty}$ together with the corresponding estimates. Consequently $\eta\in W_{per}^{k+1}$, $\kappa \in W_{per}^{k,\infty}$ and $\gamma\in W_{per}^{k+2,\infty}$. Then $\psi_1^{\ast}\in W_{per}^{k,\infty}$ and finally $\psi_0^\ast\in W_{per}^{k-1,\infty}$.

(ii) The proof proceeds analogously starting with the variations of $E$ in~\eqref{variationEgamma} and~\eqref{variationEeta}. By Lemma~\ref{lemma:AdmissibleConstraint}, the constraints (L), (A) and (S) are admissible constraints and also the geometric constraint (G) which combines (L) and (A) is admissible if $\gamma$ does not have constant curvature. Thus we may use Theorem~\ref{Deimlingc9s26t1} with $\Phi = E$, $B_r\subset X$ as in Proposition~\ref{RegularityELM} and the constraint
\begin{align*}
 F\colon X \to \R \times \R \times H_{per}^1([0,2\pi]) = Y\,,\quad (\varphi, \psi)\mapsto (G,S)(\varphi, \psi) =  (L(\varphi), A(\varphi), |\psi|^2-1)\,.
\end{align*}
Thus there exists a Lagrange multiplyer $y^\ast \in Y^\ast$ such that $E'(\gamma,\eta) + (F'(\gamma, \eta))^\ast y^\ast = 0$ in $X^\ast$. The multiplier $y^\ast$ is now given by a triple $(\ell, a, \psi^\ast)$ with $\ell$, $a\in \R$ and $\psi^\ast \in H_{per}^{-1}$. The variation with respect to $\gamma$ contains from the constraint on the length the term $\ell\int_\gamma \kappa \langle \varphi,\nu\rangle\mathrm{d}s = - \ell\int_\gamma \langle\tau,\partial_s \varphi\rangle\mathrm{d}s$. This term is also present in the variation of $E_{LM}$ (with a constant one in front of it) and we conclude as before.
\end{proof}

\section{Conclusions}\label{DynamicalProblem}
In this article, we discussed existence of minimizers for the variational problem involving the energy $E_{LM}$, derived the Euler-Lagrange system, and proved that critical points $(\gamma,\eta)$, and therefore in particular minimizers, are smooth if $\gamma$ is parametrized proportional to arc length. The main motivation for the formulation of $E_{LM}$ is the geometric functional $E_{LM}^g$ for simple plane curves where one can interpret $\eta$ as a vector field along the curve and where the surface gradient and the surface divergence are the usual geometric objects. One of the advantages of the formulation $E_{LM}$ is the fact that the variation of the energy with respect to the curve can be calculated without modeling assumptions concerning the vector field. More precisely, if one interprets $\eta$ as a vector field on the trace $\Gamma$ of a simple closed curve $\gamma$, then a variation of $\gamma$ changes its trace and one needs an extension of the vector field $\eta$ to a neighborhood of $\Gamma$. The usual approach for the derivation of the variation of $E_{LM}^g$ is to consider normal variations $\gamma_\varepsilon = \gamma + \varepsilon \varphi \nu$ of $\gamma$ and to extend $\eta$ as a constant vector field along the trace $\varepsilon \mapsto \gamma_\varepsilon(x)$ for all $x\in [0,2\pi]$. The calculation for normal variations is carried out in the Appendix and leads to the following notion of solution for the negative $L^2$--gradient flow. A family of smooth and regular plane closed curves $\gamma:[0,T]\times [0,2\pi]\to\mathbb{R}^2$ 
and a family of smooth vector fields $\eta:[0,T]\times [0,2\pi]\to\mathbb{R}^2$ is said to be a smooth solution of the $L^2$--gradient flow dynamics of the functional $E_{LM}$ if 
\begin{equation*}%\label{motioneq}
\begin{cases}
\partial_t \gamma^\perp  = \bsqb{ - \partial_{ss}(\kappa + \delta
\diverg_\gamma\eta)
+   \delta \partial_s [(\kappa + \delta \diverg_\gamma\eta)
\scp{\partial_s \eta }{\nu}]
- \frac{1}{2}  (\kappa  + \delta\diverg_\gamma \eta)^2 \kappa  -
\frac{\lambda}{2} |\partial_s \eta|^2 \kappa + \kappa}\nu\,, \\
\partial_t \eta  =
\lambda \partial_{ss} \eta + \delta  \nabla_s (\kappa +\delta
\diverg_\gamma  \eta) +\delta (\kappa + \delta \diverg_\gamma\eta) \kappa
\nu\,,
\end{cases}
\end{equation*}
see Lemma~\ref{Gradientflow}, where $\partial_t \gamma^\perp $ denotes the normal component of the velocity vector $\partial_t \gamma$. During the derivation we also collect some useful evolution equations of geometric quantities. The analysis will be addressed in a forthcoming publication~\cite{BDMP}.

A second generalization concerns the formulation of the energy $E_{LM}^g$ for two-dimensional embedded or immersed manifolds in $\mathbb{R}^3$. Let $\varphi:\Sigma\to\mathbb{R}^3$ be a smooth immersion of a $2$-dimensional orientable closed surface $\Sigma$. The Laradji-Mouritsen model~\cite{LaradjiMouritsen2000} for the energy of a liquid-liquid interface $\Sigma$ with mean curvature $H$, surfactant direction $\eta$, and material constants $\delta$, $\lambda>0$ leads to the energy functional 
\begin{align*}
  E_{LM}(\varphi, \eta) = \frac{1}{2}\int_\Sigma (H + \delta \diverg_\varphi\eta)^2 \dv{\mu}_\varphi + \frac{\lambda}{2} \int_\Sigma |\nabla_\varphi \eta|^2 \dv{\mu}_\varphi \,,
\end{align*}
where $\nabla_\varphi$ and $\diverg_{\varphi}$ are the surface gradient and the surface divergence, and where $\dv{\mu}_\varphi $ is the volume measure on $\Sigma$ induced by $\varphi$. Also for this model the evolution equations can be derived. First analytical results on the evolution can be found in~\cite{BrandThesis} and will be presented in detail in a forthcoming publication~\cite{BDMP}.

% \subsection*{Author contributions}
% 
% All authors share equal responsibility for the content of the paper.
% 
% \subsection*{Financial disclosure}
% 
% None reported.
% 
% \subsection*{Conflict of interest}
% 
% The authors declare no potential conflict of interests.

% \section*{Supporting information}

\appendix
\section{Formulas}
In the appendix, we collect formulas that are used in the text and sketch their proofs.

As indicated in the introduction, reparametrization of regular curves in Sobolev classes leaves the Sobolev class invariant. 

In fact, for $\gamma\in H^2([0,2\pi];\mathbb{R}^2)$ the function $s(x) = \int_0^x \vert \partial_x \gamma(y) \vert\ \mathrm{d}y$ satisfies $\partial_x s(x) = \vert \partial_x \gamma(x) \vert$ and $\partial_{xx} s(x) = \langle \partial_x\gamma,\partial_{xx}\gamma\rangle/\vert \partial_x \gamma(x) \vert$, thus $s$ defines an $H^2$ function which, by embedding theorems, is a diffeomorphism onto its range. From the explicit formula $\partial_y s^{-1}(y) = 1/s'(x)$ with $y=s(x)$ one sees that $s^{-1}$ is in fact in $H^2$ and 
$\partial_y (\gamma\circ s^{-1})(y) = (\partial_x\gamma)(s^{-1}(y))\cdot \partial_y(s^{-1})(y)$, $\partial_{xx} (\gamma\circ s^{-1})(y) = (\partial_{yy}\gamma)(s^{-1}(y)) \cdot(\partial_y(s^{-1})(y))^2 + (\partial_x\gamma)(s^{-1}(y))\cdot \partial_{yy}(s^{-1})(y)$. The last expression is in $H^2$ since $H^1\hookrightarrow C^0$ in one spatial dimension.

\begin{lemma}[Variation of $E$]\label{ELEqn}
The variation of the functional $E:H_{imm}^2([0,2\pi];\R^2)\times H_{per}^1([0,2\pi];\R^2)\to \R$ in a point $(\gamma,\eta)\in H_{imm}^2([0,2\pi];\R^2)\times H_{per}^1([0,2\pi];\R^2)$ in the direction $(\varphi, \psi)\in H_{per}^2([0,2\pi];\R^2)\times H_{per}^1([0,2\pi];\R^2)$ is given by
\begin{align}
\frac{\delta E}{\delta \gamma}[\varphi] & = \int_{\gamma}  (\kappa + \delta \diverg_\gamma\eta) 
\bsqb{ \scp{\partial_{ss}\varphi}{\nu} + 
\delta \scp{ \partial_s\eta  }{ \partial_s \varphi}  } \dv{s} \nonumber \\& + \int_{\gamma} \Bigl(-\frac{3}{2} (\kappa  + \delta\diverg_\gamma \eta)^2
-\frac{\lambda}{2}\vert\partial_s\eta\vert^2 \Bigr)\left\langle\tau,\partial_s\varphi\right\rangle \dv{s} \,,\label{variationEgamma} \\
\frac{\delta E}{\delta \eta}[\psi] & = \int_{\gamma} \delta (\kappa +\delta \diverg_\gamma  \eta) \diverg_\gamma\psi \dv{s}  + \lambda \int_{\gamma} \left\langle\partial_s \eta, \partial_s\psi\right\rangle \dv{s} \,.\label{variationEeta} 
\end{align}
\end{lemma}

\begin{proof}
 For $\varphi$, $\psi\in C_{per}^\infty([0,2\pi];\R^2)$ we consider variations of the curve  $\gamma$ of the form $\gamma_\varepsilon = \gamma + \varepsilon \varphi$ and of the vector field $\eta$ of the form $\eta_\varepsilon=\eta+\varepsilon\psi$. In view of the embedding $H_{imm}^2\hookrightarrow C^{1}$ the immersion $\gamma$ satisfies $|\partial_x\gamma|\geq c_0>0$ for a positive constant $c_0$ and $\gamma+\varepsilon\varphi$ is an immersion for $|\varepsilon|>0$ small enough.From~\cite[Ex.~12,~p.~25]{DoCarmoDiffGeo1976}
\begin{align*}
\kappa_\varepsilon = \frac{\det(\partial_x(\gamma+\varepsilon\varphi),\partial_{xx}(\gamma+\varepsilon\varphi))}{|\partial_x (\gamma+\varepsilon\varphi)|^3}
\end{align*}
we get with direct computations
\begin{align*}
 \frac{\mathrm{d}}{\mathrm{d}\varepsilon}\Big|_{\varepsilon = 0 } \kappa_\varepsilon 
& = \frac{\det( \partial_x \varphi, \partial_{xx}\gamma )}{|\partial_x\gamma|^3}+ \frac{\det(\partial_x\gamma, \partial_{xx}\varphi)}{|\partial_{x}\gamma|^3} - 3 \frac{\det(\partial_{x}\gamma,\partial_{xx}\gamma)}{|\partial_{x}\gamma|^5} \scp{\partial_{x}\gamma}{\partial_{x}\varphi} = I + II + III\,.
\end{align*}
We simplify the three terms with $\partial_x \tau = \kappa |\partial_x\gamma|\nu$ and $\nu = J\tau$ and the identities
\begin{align*}
\partial_{xx}\gamma & = \partial_{x} (|\partial_{x}\gamma|\tau) = \kappa |\partial_{x}\gamma|^2 \nu + \frac{1}{|\partial_{x}\gamma|} \scp{\partial_{x}\gamma}{\partial_{xx}\gamma}\tau\,,\quad \frac {\partial_{xx}\gamma}{|\partial_x \gamma|^2} = \kappa \nu + \frac{\scp{\partial_x \gamma}{\partial_{xx}\gamma}}{|\partial_x\gamma|^3}\tau \,, \\
\partial_{xx}\varphi & = \partial_{x} (|\partial_{x}\gamma|\partial_s \varphi) = |\partial_{x}\gamma|^2 \partial_{ss}\varphi + \frac{1}{|\partial_{x}\gamma|} \scp{\partial_{x}\gamma}{\partial_{xx}\gamma}\partial_s\varphi\,,
\end{align*}
according to 
\begin{align*}
I & = \frac{\det(\partial_{x}\varphi, \partial_{xx}\gamma)}{|\partial_{x}\gamma|^3}  = \kappa\scp{\partial_s \varphi}{\tau} - \frac{\scp{\partial_{x}\gamma}{\partial_{xx}\gamma}}{|\partial_{x}\gamma|^3}\scp{\partial_s \varphi}{\nu}\,, \\
II & = \frac{\det(\partial_{x}\gamma, \partial_{xx}\varphi)}{|\partial_{x}\gamma|^3}  = \scp{\partial_{ss} \varphi}{\nu} +\frac{\scp{\partial_{x}\gamma}{\partial_{xx}\gamma}}{|\partial_{x}\gamma|^3}\scp{\partial_s \varphi}{\nu}\,, \\
III & = - 3 \frac{\det(\partial_{x}\gamma,\partial_{xx}\gamma)}{|\partial_{x}\gamma|^5} \scp{\partial_{x}\gamma}{\partial_{x}\varphi}  = 
-3\kappa\scp{\tau}{\partial_{s}\varphi} \,,
\end{align*}
and the sum of all terms is just the variation of the curvature and leads to the term $\scp{\partial_{ss}\varphi}{\nu} -2\kappa\scp{\partial_s\varphi}{\tau}$. Moreover
\begin{align*}
 \frac{\mathrm{d}}{\mathrm{d}\varepsilon}\Big|_{\varepsilon = 0 } \diverg_{\gamma_\varepsilon}\eta & =  \frac{\mathrm{d}}{\mathrm{d}\varepsilon}\Big|_{\varepsilon = 0 } \Bigscp{ \frac{\partial_x\eta}{|\partial_x \gamma_\varepsilon|} }{\frac{\partial_x \gamma_\varepsilon}{|\partial_x \gamma_\varepsilon|}}=\Bigscp{ \frac{\partial_x \eta}{|\partial_x\gamma|} }{\frac{\partial_x \varphi}{|\partial_x\gamma|} } - 2\scp{\partial_x \eta}{\partial_x \gamma} \frac{\scp{\partial_x\gamma}{\partial_x\varphi}}{|\partial_x\gamma|^4} \\
& =\scp{ \partial_s \eta}{ \partial_s \varphi} - 2 (\diverg_\gamma \eta) \scp{\tau}{\partial_s \varphi} \,,\\
 \frac{\mathrm{d}}{\mathrm{d}\varepsilon}\Big|_{\varepsilon = 0 }  | \partial_s \eta |^2 & =  \frac{\mathrm{d}}{\mathrm{d}\varepsilon}\Big|_{\varepsilon = 0 } \frac{|\partial_x\eta|^2}{|\partial_x\gamma_\varepsilon|^2} =-2 \frac{|\partial_x\eta|^2}{|\partial_x\gamma_\varepsilon|^4} \scp{\partial_x \gamma}{\partial_x\varphi} =-2 |\partial_s \eta|^2 \scp{\tau}{\partial_s \varphi}  \,.
\end{align*}
Putting all together we get 
\begin{align*}
\frac{\mathrm{d}}{\mathrm{d}\varepsilon} E(\gamma_\varepsilon,\eta)\Big|_{\varepsilon = 0 }  &= \int_\gamma  (\kappa + \delta \diverg_\gamma\eta) 
\bsqb{ \scp{\partial_{ss}\varphi}{\nu} + 
\delta \scp{ \partial_s\eta  }{ \partial_s \varphi}  } \dv{s} \\&  + \int_\gamma \Bigl(-\frac{3}{2} (\kappa  + \delta\diverg_\gamma \eta)^2
-\frac{\lambda}{2}\vert\partial_s\eta\vert^2
\Bigr)\left\langle\tau,\partial_s\varphi\right\rangle \dv{s} \,.\end{align*}
For the variation in $\eta$ we find with an integration by parts
\begin{align*}
\frac{\mathrm{d}}{\mathrm{d}\varepsilon} E(\gamma,\eta_\varepsilon)\Big|_{\varepsilon = 0 } & =
 \frac{\mathrm{d}}{\mathrm{d}\varepsilon} \Big|_{\varepsilon = 0 }
 \int_\gamma \frac{1}{2}(\kappa + \delta \diverg_\gamma \eta_\varepsilon)^2 \dv{s} + \frac{\lambda}{2}\int_\gamma |\partial_s  \eta_\varepsilon|^2 \dv{s} \\
 &= \int_\gamma \delta (\kappa +\delta \diverg_\gamma  \eta) \diverg_\gamma\psi \dv{s}  + \lambda \int_\gamma \left\langle\partial_s \eta, \partial_s\psi\right\rangle \dv{s} \,.\end{align*}
By approximation, the necessary conditions hold for $\varphi\in H_{per}^2$ and $\psi\in H_{per}^1$.
\end{proof}

\begin{lemma}[Euler-Lagrange equations for $E_{LM}$]\label{ELEqnELM}
The variation of 
$$
E_{LM}:H_{imm}^2([0,2\pi];\R^2)\times H_{per}^1([0,2\pi];\R^2)\to \R
$$ 
in a point $(\gamma,\eta)\in H_{imm}^2([0,2\pi];\R^2)\times H_{per}^1([0,2\pi];\R^2)$ in the direction $(\varphi, \psi)\in H_{per}^2([0,2\pi];\R^2)\times H_{per}^1([0,2\pi];\R^2)$ is given by
\begin{align}
\frac{\delta E_{LM}}{\delta \gamma}[\varphi] & = \int_{\gamma}  (\kappa + \delta \diverg_\gamma\eta) 
\bsqb{ \scp{\partial_{ss}\varphi}{\nu} + 
\delta \scp{ \partial_s\eta  }{ \partial_s \varphi}  } \dv{s} \nonumber\\&  + \int_{\gamma} \Bigl(-\frac{3}{2} (\kappa  + \delta\diverg_\gamma \eta)^2
-\frac{\lambda}{2}\vert\partial_s\eta\vert^2+1
\Bigr)\left\langle\tau,\partial_s\varphi\right\rangle \dv{s} \,,\label{variationELMgamma} \\
\frac{\delta E_{LM}}{\delta \eta}[\psi] & = \int_{\gamma} \delta (\kappa +\delta \diverg_\gamma  \eta) \diverg_\gamma\psi \dv{s}  + \lambda \int_{\gamma} \left\langle\partial_s \eta, \partial_s\psi\right\rangle \dv{s} \,.\label{variationELMeta} 
\end{align}
\end{lemma}

\begin{proof}
 The only difference in the variation of $E_{LM}$ compared to the variation of $E$ is the additional term $L(\gamma)$ which leads to an additional term of $\int_\gamma\langle \tau,\partial_s \varphi\rangle\dv{s}$ in the variation.
\end{proof}

\begin{lemma}[Gradient flow for $E_{LM}$]\label{Gradientflow}
Let $(\gamma_0,\eta_0)\in H^4_{imm}([0,2\pi];\mathbb{R}^2)\times  H^3_{per}([0,2\pi];\mathbb{R}^2)$ and let $T>0$. 
Suppose that  $\gamma:[0,T]\times [0,2\pi]\to\mathbb{R}^2$
is a  time-dependent family of regular plane closed curves
at least of class $H^1$ in the time variable and $H^4$ in the space variable
and that $\eta:[0,T]\times [0,2\pi]\to\mathbb{R}^2$ is a time-dependent family
of vector fields of class $H^1$ in the time variable and $H^3$ in the space variable.
Then $(\gamma,\eta)$ is a solution to the (formal) $L^2$--gradient flow of $E_{LM}$
(obtained considering only normal variations of the curve)
in the time interval $[0,T]$ with initial datum $(\gamma_0,\eta_0)$
if and only if $(\gamma,\eta)$
satisfy the following system for all $t\in[0,T]$, $x\in [0,2\pi]$:
\begin{equation}\label{motioneq}
\begin{cases}
\partial_t \gamma^\perp  = \bsqb{ - \partial_{ss}(\kappa + \delta
\diverg_\gamma\eta)
+   \delta \partial_s [(\kappa + \delta \diverg_\gamma\eta)
\scp{\partial_s \eta }{\nu}]
- \frac{1}{2}  (\kappa  + \delta\diverg_\gamma \eta)^2 \kappa  -
\frac{\lambda}{2} |\partial_s \eta|^2 \kappa + \kappa}\nu\,, \\
\partial_t \eta  =
\lambda \partial_{ss} \eta + \delta  \nabla_s (\kappa +\delta
\diverg_\gamma  \eta) +\delta (\kappa + \delta \diverg_\gamma\eta) \kappa
\nu\,,\\
(\gamma(0,x),\eta(0,x))=(\gamma_0,\eta_0)
\end{cases}
\end{equation}
where $\partial_t \gamma^\perp $ denotes the normal component of the velocity vector.
\end{lemma}

\begin{proof}
We characterize, at least formally, the $L^2$--gradient flow dynamics
of $E_{LM}$.

For $\varphi$, $\psi\in C_{per}^\infty([0,2\pi];\R^2)$ and $|\epsilon|>0$ small enough, we consider variations of the vector field $\eta$ of the form
$\eta_\varepsilon=\eta+\varepsilon\psi$
and
of the curve  $\gamma$ of the form
$\gamma_\varepsilon = \gamma + \varepsilon \varphi$,
and we require that
$\gamma_\varepsilon$ be a normal variation of $\gamma$ with $\varphi=u\nu$, where $\nu = J\tau$
is the unit normal vector to $\gamma$, and hence
\begin{align*}
\varphi =u\nu\,,\quad \partial_s\varphi = \partial_su\nu-u\kappa\tau\,,\quad  \partial^2_s\varphi =\left(\partial^2_su-u\kappa^2\right)\nu
-\left(2\partial_s u\kappa + u\partial_s \kappa\right)\tau\,.
\end{align*}
In order to pass from the variations obtained in Lemma~\ref{ELEqn} to expressions that do not involve derivatives of $\varphi$ and $\psi$, we need to integrate by parts and this calculation needs additional regularity for $\gamma$ and $\eta$ compared to Theorem~\ref{regularity-stationary-points}. From~\eqref{variationELMgamma}
\begin{align*}
&\left.\frac{\mathrm d} {{\mathrm d} \varepsilon}
E_{LM}(\gamma_\varepsilon, \eta)\right|_{\varepsilon=0}
=\int_{\gamma} \left( \kappa+\delta\diverg_\gamma \eta\right)\partial_{ss} u
-\kappa^2\left( \kappa+\delta\diverg_\gamma \eta\right) u\\
&-\kappa\left(-\frac{3}{2}\left(\kappa+\delta\diverg_\gamma \eta\right)^2
-\frac{\lambda}{2}\vert\partial_s\eta\vert^2+1\right)u
+\delta\left(\kappa+\delta\diverg_\gamma \eta\right)\left\langle\partial_s\eta,\partial_su\nu-u\kappa\tau\right\rangle\dv{s}\\
&=\int_{\gamma} \left( \kappa+\delta\diverg_\gamma \eta\right)\partial_{ss} u
+\delta\left(\kappa+\delta\diverg_\gamma \eta\right)\left\langle\partial_s\eta,\nu\right\rangle\partial_su
+\left(\frac{1}{2}\left(\kappa+\delta\diverg_\gamma \eta\right)^2
+\frac{\lambda}{2}\vert\partial_s\eta\vert^2-1\right)\kappa
u\dv{s}\,.
\end{align*}
Integrating by parts we obtain
\begin{align*}
\left.\frac{\mathrm d} {{\mathrm d} \varepsilon}
E_{LM}(\gamma_\varepsilon, \eta)\right|_{\varepsilon=0}
&=\int_{\gamma} \partial_{ss}\left( \kappa+\delta\diverg_\gamma \eta\right) u
-\delta\partial_s\left[\left(\kappa+\delta\diverg_\gamma \eta\right)\left\langle\partial_s\eta,\nu\right\rangle\right]u\\
&+\left(\frac{1}{2}\left(\kappa+\delta\diverg_\gamma \eta\right)^2
+\frac{\lambda}{2}\vert\partial_s\eta\vert^2-1\right)\kappa
u\dv{s}\,,
\end{align*}
and we rearrange terms to obtain 
%\rot{Vorzeichen??? Was soll geschrieben werden?}
\begin{align}\label{derivativegamma}
\frac{\partial E_{LM}}{\partial \gamma}(\gamma,\eta)[\varphi]
=\int_{\gamma} \left\langle \bsqb{\partial_{ss}(\kappa + \delta
\diverg_\gamma\eta)
-   \delta \partial_s [(\kappa + \delta \diverg_\gamma\eta)
\scp{\partial_s \eta }{\nu}]\right.\nonumber\\
\left.
+ \frac{1}{2}  (\kappa  + \delta\diverg_\gamma \eta)^2 \kappa  +
\frac{\lambda}{2} |\partial_s \eta|^2 \kappa -
\kappa}\nu,\varphi\right\rangle\dv{s}\,.
\end{align}
Concerning $\eta$ we integrate by parts in~\eqref{variationELMeta} and get 
%\begin{align*}
%& \frac{\mathrm{d}}{\mathrm{d}\varepsilon}
%E_{LM}(\gamma,\eta_\varepsilon)\Big|_{\varepsilon = 0 }
%=
%- \int_\gamma \delta\scp{ \nabla_s (\kappa +\delta \diverg_\gamma  \eta)}{\psi}
%\dv{s} - \int_\gamma (\kappa + \delta \diverg_\gamma\eta) \kappa \scp{\psi}{\nu}\dv{s}
%- \lambda \int_\gamma \scp{\partial_{ss} \eta}{ \psi} \dv{s}\,,
%\end{align*}
%and we deduce the equation \rot{Vorzeichen!}
\begin{align}\label{derivativeeta}
\frac{\partial E_{LM}}{\partial \eta}(\gamma,\eta)[\psi]=
- \int_\gamma \left\langle
\lambda \partial_{ss} \eta + \delta  \nabla_s (\kappa +\delta
\diverg_\gamma  \eta) +\delta (\kappa + \delta \diverg_\gamma\eta) \kappa
\nu,\psi\right\rangle\dv{s}\,.
\end{align}
We proceed as in Theorem~\ref{RegularityELM} and consider $H_{imm}^4$ as  open subset in $H_{per}^4$. There exists an $r>0$ with $B(\gamma,r)\subset H_{imm}^4\subset H_{per}^4$. 
The Gateaux derivatives in the directions $\varphi\in H^4_{per}$ and $\psi\in H_{per}^3$ are given by~\eqref{derivativegamma} and by~\eqref{derivativeeta}. 
It is easy to see that the partial Fr\'echet derivatives exist and are 
continuous on $B((\gamma, \eta),r)$,
thus the map $E_{LM}$ is continuously Fr\'echet differentiable in a neighborhood of an immersion $\gamma$ and for $(\varphi,\psi)\in H_{per}^4\times H_{per}^3$ we have the representation 
\begin{equation*}
E'(\gamma,\eta)[\varphi,\psi]=
 \left.\frac{\mathrm d} {{\mathrm d} \varepsilon}\right|_{\varepsilon=0}
E(\gamma + \varepsilon \varphi, \eta+\varepsilon\psi)\,.
\end{equation*}
Therefore $E'_{LM}$ is the gradient of $E_{LM}$
and if a time-dependent family $(\gamma_t, \eta_t)_{t\in [0,T)}$ moves with velocity equal to the negative gradient of $E_{LM}$, then this family is a solution of the associated gradient flow.
\end{proof}

We also prove that the energy $E_{LM}$ decreases along
the flow.
To do so we introduce some notation:
\begin{align*}
 z&=\kappa+\delta\diverg_\gamma \eta\\
V&=- \partial_{ss}z
+   \delta \partial_s ( z \scp{\partial_s \eta }{\nu} )
- \frac{1}{2} z^2 \kappa  - \frac{\lambda}{2} |\partial_s \eta|^2
\kappa + \kappa\,, \\
\vec{W} & =
\lambda \partial_{ss} \eta + \delta  \partial_s z\tau +\delta z
\kappa \nu
\,.
\end{align*}

We give here some formulas that describe the evolution of
geometric quantities under the gradient flow of $E_{LM}$.

\begin{lemma}
Suppose that $\gamma_t=\gamma_t^\perp$.
We have
\begin{subequations}\label{UsefulFormula}
\begin{align}\label{UsefulFormulaA}
\partial_t\left(\mathrm{d}s\right)&=-\kappa V\mathrm{d}s\,,\\
\label{UsefulFormulaB}
\partial_t\partial_s\cdot&=\partial_s\partial_t\cdot+\kappa
V\partial_s\cdot\,,\\
\label{UsefulFormulaC}
\partial_t\tau&=\partial_t\partial_s\gamma
=\partial_s\partial_t\gamma+\kappa V\partial_s\gamma=\partial_sV\nu\,,\\
\label{UsefulFormulaD}
\partial_t\kappa &=\partial_{ss}V+V\kappa^2\,,\\
\label{UsefulFormulaE}
\partial_t\partial_s\eta&=\partial_s\partial_t\eta+kV\partial_s\eta=\partial_s\vec{W}+\kappa
V\partial_s\eta\,,\\
\label{UsefulFormulaF}
\partial_t\diverg_\gamma \eta&
=\langle\partial_s\vec{W},\tau\rangle
+\kappa
V\diverg_\gamma \eta+\left\langle\partial_s  \eta,\partial_sV\nu\right\rangle\,,\\
\label{UsefulFormulaG}
\partial_t z&=\partial_{ss}V+\kappa V
z+\delta\langle\partial_s\vec{W},\tau\rangle+\delta\left\langle\partial_s    \eta,\partial_sV\nu\right\rangle\,.
\end{align}
 
\end{subequations}

\end{lemma}
\begin{proof}
Formulas~\eqref{UsefulFormulaA}--\eqref{UsefulFormulaD} have been derived several times in the literature, we refer for instance to~\cite[Lemma 2.1]{DzuikKuwertSchaetzle2002}.
Let us pass to compute the evolution equation for $\partial_s \eta$.
Thanks to~\eqref{UsefulFormulaB}
$\partial_t\partial_s\eta=\partial_s\partial_t \eta +\kappa V \partial_s\eta$.
Since the equation of motion for $\eta$ reads $\partial_t\eta =\vec{W}$,
we can conclude $\partial_t\partial_s\eta=\partial_s\vec{W} +\kappa V \partial_s\eta$, that is~\eqref{UsefulFormulaE}. With this formula we prove~\eqref{UsefulFormulaF} since
\begin{align*}
\partial_t\diverg_\gamma \eta&
=\left\langle\partial_t\partial_s\eta, \tau\right\rangle
+\left\langle\partial_s\eta, \partial_t\tau\right\rangle\\
&=\langle\partial_s\vec{W},\tau\rangle
+\left\langle\kappa V\partial_s\eta,\tau\right\rangle
+\left\langle\partial_s\eta, \partial_sV\nu\right\rangle
=\langle\partial_s\vec{W},\tau\rangle
+\kappa V\diverg_\gamma \eta+\left\langle\partial_s  \eta,\partial_sV\nu\right\rangle\,.
\end{align*}
Finally, combining~\eqref{UsefulFormulaD} and~\eqref{UsefulFormulaE} we get
\begin{align*}
\partial_t z
&=\partial_t\left(\kappa+\delta\diverg_\gamma\eta\right)=
\partial_{ss}V+V\kappa^2+\delta(\langle\partial_s\vec{W},\tau\rangle
+\kappa V\diverg_\gamma \eta+\left\langle\partial_s  \eta,\partial_sV\nu\right\rangle) \\
&=\partial_{ss}V+\kappa V
z+\delta\left\langle\partial_s\vec{W},\tau\right\rangle+\delta\langle\partial_s \eta,\partial_sV\nu\rangle\
\end{align*}
and this is~\eqref{UsefulFormulaG}.
\end{proof}

\begin{lemma}\label{decreasing-in-time}
Let $(\gamma_t,\eta_t)$ be a time dependent family of closed curves and
vector fields
evolving under the law~\eqref{motioneq} with $\gamma_t=\gamma_t^\perp$. Then
\begin{equation*}
\partial_t E_{LM}(\gamma_t,\eta_t)=\int_{\gamma_t}
-V^2-\vert\vec{W}\vert^2\dv{s}\,.
\end{equation*}
\end{lemma}

\begin{proof}
With an extensive use of the formulas of the previous lemma, we can
compute 
\begin{align*}
&\quad\frac{\mathrm{d}}{\mathrm{d}t}\left[\int_{\gamma_t}
\frac{1}{2}\left(\kappa+\delta\diverg_\gamma \eta\right)^2
+\frac{\lambda}{2}\vert\partial_s\eta\vert^2
+1\dv{s}\right]=\frac{\mathrm{d}}{\mathrm{d}t}\left[\int_{\gamma_t}
\frac{1}{2}z^2+\frac{\lambda}{2}\vert\partial_s\eta\vert^2
+1\dv{s}\right]\\
&=\int_{\gamma_t} \partial_{ss}Vz+\kappa V
z^2+\delta\langle\partial_s\vec{W},\tau\rangle
z+\delta\left\langle\partial_s  \eta,\partial_sV\nu\right\rangle z
+\lambda\langle\partial_s\eta, \partial_s\vec{W}\rangle
+\lambda\left\langle\partial_s\eta, \kappa V\partial_s\eta\right\rangle\\
&-\frac{1}{2}\kappa V z^2
-\frac{\lambda}{2} \kappa V\vert\partial_s\eta\vert^2-\kappa
V\dv{s}\\
&=\int_{\gamma_t} \partial_{ss}zV+\frac{1}{2}\kappa z^2 V
+\frac{\lambda}{2} \kappa \vert\partial_s\eta\vert^2V-\kappa
V-\delta\partial_s\left(\left\langle\partial_s\eta,\nu\right\rangle
z\right)V\\
&-\lambda\langle\partial_{ss}\eta, \vec{W}\rangle
-\delta\langle\partial_s(z)\tau,\vec{W}\rangle
-\delta\langle z\kappa\nu,\vec{W}\rangle\dv{s}\\
&=\int_{\gamma_t} -V^2-\langle\vec{W},\vec{W}\rangle\dv{s}\,,
\end{align*}
as desired.
\end{proof}

\begin{corollary}
Let $\gamma:[0,T]\times [0,2\pi]\to\mathbb{R}^2$
be a family of time-dependent smooth and regular plane closed curves
and $\eta:[0,T]\times [0,2\pi]\to\mathbb{R}^2$ a time-dependent family
of vector fields
evolving under the law~\eqref{motioneq} with  $\gamma_t=\gamma_t^\perp$ in the time interval $[0,T]$
with initial datum $(\gamma_0(x),\eta_0(x))=(\gamma(0,x),\eta(0,x))$.
Then
for every $t\in [0,T]$ the energy $E_{LM}(\gamma(t,x),\eta(t,x))$
at time $t$ is bounded by the energy of the initial
datum
$E_{LM}(\gamma_0,\eta_0)$.
\end{corollary}
\begin{proof}
Suppose that $(\gamma_t,\eta_t)$ is a solution to system~\eqref{motioneq} in the time interval $[0,T]$
with initial datum $(\gamma_0(x),\eta_0(x))=(\gamma(0,x),\eta(0,x))$.
Thanks to Lemma~\ref{decreasing-in-time}
we have that $\partial_t E$ is nonpositive, so $E$ is decreasing in $t$, 
the maximum of the energy is attained at $t=0$
and for all $t\in(0,T]$ the energy of $(\gamma_t,\eta_t)$ is less or equal 
to the energy of the initial datum.
\end{proof}

\begin{remark}
If we combine the corollary with Lemma~\ref{bound-geo-quantities}
we also get a (uniform in time) bound on the $L^2$--norm of the curvature
of evolving curves $\gamma_t$
and  a (uniform in time) bound
on the $L^2$--norm of $\partial_s\eta_t$ and on $\diverg(\eta_t)$.
\end{remark}

\bibliographystyle{amsplain}
\bibliography{Brand-Dolzmann-Pluda}%

\end{document}